\documentclass{amsart}
\usepackage{amssymb}
\usepackage{verbatim}
\newtheorem{theorem}{Theorem}[section]
\newtheorem{lemma}[theorem]{Lemma}
\newtheorem{proposition}[theorem]{Proposition}

\newtheorem{corollary}[theorem]{Corollary}

\theoremstyle{definition}
\newtheorem{definition}[theorem]{Definition}

\theoremstyle{remark}
\newtheorem{remark}[theorem]{Remark}
\numberwithin{equation}{section}

\newcommand{\LP}[2]{L^{#1}(X_{#2}, \Sigma_{#2}, m_{#2})}
\newcommand{\NCLP}[2]{L^{#1}(\mathcal{M}_{#2}, \tau_{#2})}
\newcommand{\M}{\mathcal{M}}
\newcommand{\cS}{\mathcal{S}}
\newcommand{\B}{\mathcal{B}}
\newcommand{\tM}{\widetilde{\mathcal{M}}}

\newcommand{\I}{1\!{\mathrm l}}

\newcommand{\bR}{{\mathbb{R}}}

\newcommand{\bN}{{\mathbb{N}}}

\begin{document}

\title{Maps on Noncommutative Orlicz Spaces}

\author{Louis E. Labuschagne}
\address{Department of Maths, Applied Maths and Astronomy,
P.O.Box 392, University of South Africa, 0003 Pretoria, South
Africa} 
\email{labusle@unisa.ac.za}
\author{W{\l}adys{\l}aw A. Majewski}
\address{Institute of Theoretical Physics and Astrophysics, Gda{\'n}sk
University, Wita Stwo\-sza~57, 80-952 Gda{\'n}sk, Poland} 
\email{fizwam@univ.gda.pl}
\subjclass[2000]{46L52, 47B33 (Primary); 47L90 (Secondary)}
\date{\today}
\keywords{}

\begin{abstract} 
A generalization of the Pistone-Sempi argument, demonstrating the utility of non-commutative Orlicz spaces, is presented. The question of lifting positive maps defined on von Neumann algebra to maps on corresponding noncommutative Orlicz spaces is discussed. In particular, we describe those Jordan $*$-morphisms on semifinite von Neumann algebras 
which in a canonical way induce quantum composition operators on noncommutative Orlicz spaces. Consequently, it is proved that the framework of noncommutative Orlicz spaces is well suited for an analysis of a 
large class of interesting
noncommutative dynamical systems.
\end{abstract}

\maketitle
\section{Introduction}
This article is devoted to a study of maps on noncommutative Orlicz spaces. Noncommutative Orlicz spaces can be defined either in a very algebraic way (see \cite{Kun}, \cite{ARZ}) or employing Banach space geometry (see \cite{DDdP}). The second approach based on the concept of Banach Function Spaces, among other properties, readily indicates similarities with the classical origins as well as clarifies why Orlicz spaces, being a special case of rearrangement-invariant spaces, are well suited for interpolation techniques. As these features are important for our analysis, the second approach is taken.

Section 2 consists of some preliminaries and is expository. In section 3 noncommutative regular statistical models are defined
and a noncommutative generalization of the Pistone-Sempi theorem is proved. Composition operators are introduced in 
section 4. Section 5 is devoted to detailed analysis of positive maps on von Neumann algebras.
We prove that unital pure CP maps and a large class of Jordan morphisms induce bounded maps on noncommutative Orlicz 
spaces. A characterization of Jordan morphisms which induce composition operators is given in section 6. 
Since the described maps can be 
considered as archetypes of dynamical maps for noncommutative regular statistical models
both sections, $5$ and $6$, can be treated as the first step towards the foundation of the theory of noncommutative dynamical systems associated with regular statistical models.
On the other hand, the main result of section 6 (see Theorem \ref{5.1}) can be considered as a highly nontrivial
noncommutative counterpart of the Banach-Lamperti theorem.

\section{Preliminaries}

General von Neumann algebraic notation will be based on that of \cite{BRo},  
\cite{Tak} with $\M$ denoting a von Neumann algebra and $\I$ the identity  
element thereof. As regards $L_p$-spaces we will use \cite{Tp} and \cite{FK} 
as basic references for the non-commutative context. In this paper we will 
restrict attention to the case of semifinite von Neumann algebras. The fns 
trace of such an algebra $\M$ will be denoted by $\tau_{\M} = \tau$. The projection 
lattice of a von Neumann algebra $\M$ will be denoted by $\mathbb{P}(\M)$.

By the term an \emph{Orlicz function} we understand a convex function 
$\varphi : [0, \infty) \to [0, \infty]$ satisfying $\varphi(0) = 0$ and $\lim_{u \to \infty} 
\varphi(u) = \infty$, which is neither identically zero nor infinite valued on all of 
$(0, \infty)$, and which is left continuous at $b_\varphi = \sup\{u > 0 : \varphi(u) < 
\infty\}$. The constant $a_\varphi = \inf\{u > 0 : \varphi(u) > 0\}$ also plays an important 
role in studying Orlicz functions. It is worth pointing out that any Orlicz function 
must also be increasing, and continuous on $[0, b_\varphi]$. 

Each such function induces a complementary Orlicz function $\varphi^*$ which is defined by 
$\varphi^*(u) = \sup_{v > 0}(uv - \varphi(v))$. The formal ``inverse'' $\varphi^{-1}: [0, \infty) 
\to [0, \infty]$ of an Orlicz function is defined by the formula
$$\varphi^{-1}(t) = \sup\{s : \varphi(s) \leq t\}.$$Denoting this function by $\varphi^{-1}$ is of course just a notational convention for Orlicz functions. It is only really in the case where $a_\varphi = 0$ and $b_\varphi = \infty$ that it is an inverse function in the true sense of the word. (See also Lemma \ref{lemma4}.)  

Let $\LP{0}{}$ be the space of measurable 
functions on some $\sigma$-finite measure space $(X, \Sigma, m)$. The Orlicz space 
$\LP{\varphi}{}$ associated with  $\varphi$ is defined to be the set $$L^{\varphi} = \{f \in 
L^0 : \varphi(\lambda |f|) \in L^1 \quad \mbox{for some} \quad \lambda = \lambda(f) > 0\}.$$
This space turns out to be a linear subspace of $L^0$ which becomes a Banach space when 
equipped with the so-called Luxemburg-Nakano norm 
$$\|f\|_\varphi = \inf\{\lambda > 0 : \|\varphi(|f|/\lambda)\|_1 \leq 1\}.$$
An equivalent norm (the Orlicz norm in Amemiya form) is given by 
$$\|f\|^0_\varphi = \inf_{k > 0}(1 + \|\varphi(k|f|)\|_1)/k.$$ 

We say that $\varphi$ satisfies $\Delta_2$ \emph{for all} $u$ if there exists a positive 
constant $K$ such that $\varphi(2u) \leq K \varphi(u)$ for all $u > 0$. In such a case 
$$(L^\varphi, \|\cdot\|_{\varphi})^* = (L^{\varphi^*}, \|\cdot\|^0_{\varphi^*}) \quad 
\mbox{and}\quad (L^\varphi, \|\cdot\|^0_{\varphi})^* = (L^{\varphi^*}, \|\cdot\|_{\varphi^*})$$  

Let $\varphi$ be a given Orlicz function. In the context of semifinite von Neumann algebras 
$\M$ equipped with an fns trace $\tau$, the space of all $\tau$-measurable operators 
$\widetilde{\M}$ (equipped with the topology of convergence in measure) plays the role of 
$L^0$. In the specific case where $\varphi$ is a so-called Young's function (i.e. a map $\varphi: \bR \to [0, \infty]$ having the properties of Orlicz function with additional symmetry $\varphi(x) = \varphi(-x)$), Kunze \cite{Kun} 
used this identification to define the associated noncommutative Orlicz space to be 
$$\NCLP{\varphi}{} = \cup_{n=1}^\infty n \{f \in \widetilde{\M} : \tau(\varphi(|f|)) \leq 1\}$$ 
and showed that this too is a linear space which becomes a Banach space when equipped with the 
Luxemburg-Nakano norm $$\|f\|_\varphi = \inf\{\lambda > 0 : \tau(\varphi(|f|/\lambda)) \leq 
1\}.$$Using the linearity it is not hard to see that $$\NCLP{\varphi}{} = \{f \in 
\widetilde{\M} : \tau(\varphi(\lambda|f|)) < \infty  \quad \mbox{for some} \quad \lambda = 
\lambda(f) > 0\}$$Thus there is a clear analogy with the commutative case.

Given an element $f \in \widetilde{\M}$ and $t \in [0, \infty)$, the generalised singular 
value $\mu_t(f)$ is defined by $\mu_t(f) = \inf\{s \geq 0 : \tau(\I - e_s(|f|)) \leq t\}$ 
where $e_s(|f|)$, $s \in \mathbb{R}$, is the spectral resolution of $|f|$. The function $t \to 
\mu_t(f)$ will generally be denoted by $\mu(f)$. For details on the generalised singular value 
see \cite{FK}.  (This directly extends classical notions where for any $f \in \LP{\infty}{}$, 
the function $(0, \infty) \to [0, \infty] : t \to \mu_t(f)$ is known as the decreasing 
rearrangement of $f$.) We proceed to briefly review the concept of a Banach Function Space of 
measurable functions on $(0, \infty)$ (see \cite{DDdP}.) A function norm 
$\rho$ on $L^0(0, \infty)$ is defined to be a mapping $\rho : L^0_+ \to [0, \infty]$ satisfying
\begin{itemize}
\item $\rho(f) = 0$ iff $f = 0$ a.e.  
\item $\rho(\lambda f) = \lambda\rho(f)$ for all $f \in L^0_+, \lambda > 0$.
\item $\rho(f + g) \leq \rho(f) + \rho(g)$ for all .
\item $f \leq g$ implies $\rho(f) \leq \rho(g)$ for all $f, g \in L^0_+$.
\end{itemize}
Such a $\rho$ may be extended to all of $L^0$ by setting $\rho(f) = \rho(|f|)$, in which case 
we may then define $L^{\rho}(0, \infty) = \{f \in L^0(0, \infty) : \rho(f) < \infty\}$. If 
now $L^{\rho}(0, \infty)$ turns out to be a Banach space when equipped with the norm 
$\rho(\cdot)$, we refer to it as a Banach Function Space. 
If $\rho(f) \leq \lim\inf_n \rho(f_n)$ 
whenever $(f_n) \subset L^0$ converges almost everywhere to $f \in L^0$, we say that $\rho$ 
has the Fatou Property. (This is equivalent to the requirement that $\rho(f_n) \uparrow \rho(f)$ whenever $0 \leq f_n \uparrow f$ a.e. \cite[11.4]{AB}.) If less generally this implication only holds for $(f_n) \cup \{f\} 
\subset L^{\rho}$, we say that $\rho$ is lower semi-continuous. If further the situation $f 
\in L^\rho$, $g \in L^0$ and $\mu_t(f) = \mu_t(g)$ for all $t > 0$, forces $g \in L^\rho$ and 
$\rho(g) = \rho(f)$, we call $L^{\rho}$ rearrangement invariant (or symmetric). (The concept of a Banach function norm can of course in a similar fashion equally well be defined for $L^0(X, \Sigma, \nu)$, where $(X, \Sigma, \nu)$ is an arbitrary measure space.)

Using the above context Dodds, Dodds and de Pagter \cite{DDdP} formally defined the noncommutative space 
$L^\rho(\widetilde{\M})$ to be $$L^\rho(\widetilde{\M}) = \{f \in \widetilde{\M} : \mu(f) \in 
L^{\rho}(0, \infty)\}$$and showed that if $\rho$ is lower semicontinuous and $L^{\rho}(0, 
\infty)$ rearrangement-invariant, $L^\rho(\widetilde{\M})$ is a Banach space when equipped 
with the norm $\|f\|_\rho = \rho(\mu(f))$. The space $L^\rho(\widetilde{\M})$ is said to be \emph{fully 
symmetric} if for any $f \in L^\rho(\widetilde{\M})$ and $g \in \widetilde{\M}$ the property 
$$\int_0^\alpha \mu_t(|g|)\, \mathrm{d}t \leq \int_0^\alpha\mu_t(|f|)\, \mathrm{d}t \quad \mbox{for all} 
\quad \alpha > 0$$ensures that $g \in L^\rho(\widetilde{\M})$ with $\rho(g) \leq \rho(f)$. Now for any 
Orlicz function $\varphi$, the Orlicz 
space $L^\varphi(0, \infty)$ is known to be a rearrangement invariant Banach Function space 
with the norm having the Fatou Property \cite[Theorem 4.8.9]{BS}. Thus on selecting $\rho$ to be 
one of $\|\cdot\|_\varphi$ or $\|\cdot\|^0_\varphi$, the very general framework of Dodds, 
Dodds and de Pagter presents us with an alternative approach to realising noncommutative 
Orlicz spaces. We pause to show that this approach canonically contains the spaces of Kunze 
\cite{Kun}. Recall that any Orlicz function is in fact continuous, non-negative and increasing 
on $[0, b_\varphi)$. So if (as is the case in \cite{Kun}) we assume that $b_\varphi = \infty$, 
then for any $\lambda > 0$ and any $f \in \widetilde{\M}$, we always have that 
$$\tau\left(\varphi\left(\frac{1}{\lambda}|f|\right)\right) = \int_0^\infty \varphi\left(\frac{1}{\lambda}\mu_t(|f|)\right)\, 
\mathrm{d}t$$by \cite[2.8]{FK}. The equivalence of the two approaches in this setting, is a trivial 
consequence of this equality. 

However in the case where $b_\varphi < \infty$, the equivalence of the two approaches is not immediately obvious. If $b_\varphi < \infty$, then for any $f \in \widetilde{\M}$, we can always give meaning to $\varphi(\mu_t(f))$. However $\varphi(|f|)$ may not even exist as an element of $\widetilde{\M}$! We proceed to show that even here, the two approaches yield identical spaces.

\begin{lemma}\label{DPvsKlemma} Let $\varphi$ be an Orlicz function and $f \in \widetilde{\M}$ a $\tau$-measurable element for which $\varphi(|f|)$ is again $\tau$-measurable. Extend $\varphi$ to a function on $[0, \infty]$ by setting $\varphi(\infty) = \infty$. Then $\varphi(\mu_t(f)) = \mu_t(\varphi(|f|))$ for any $t \geq 0$. Moreover $\tau(\varphi(|f|)) = \int_0^\infty \varphi(\mu_t(|f|))\, \mathrm{d}t$.
\end{lemma}

\begin{proof}
The second claim will follow from the first by an application of \cite[2.7]{FK}. To prove 
the first claim we may replace $\M$ by a maximal abelian von Neumann subalgebra $\M_0$ to which 
both $|f|$ and $\varphi(|f|)$ are affiliated (see \cite[2.3(1)]{FK}). Let $e$ be any projection
 in this subalgebra. Now notice that $\sigma(|f|) \subset [0, \infty)$. 

First suppose that $\varphi$ is bounded on $\sigma(|f|e)$. (By the Borel functional calculus 
$\varphi(|f|e)$ will then of course be bounded.) Since $\lim_{u\to \infty}\varphi(u) = 
\infty$, we must then have that $\sigma(|f|e)$ itself is a bounded subset of $[0, \infty)$. 
Thus $|f|e$ must be bounded. By spectral theory for positive elements, we now have that 
$\||f|e\| = \mathrm{max}\{\lambda : \lambda \in \sigma(|f|e)\}$. Since $\varphi$ is increasing 
and non-negative on $[0, \infty)$, the Borel functional calculus also ensures that 
$$\varphi(\||f|e\|) = \mathrm{max}\{\varphi(\lambda) : \lambda \in \sigma(|f|e)\} = 
\|\varphi(|f|e)\|.$$

If $\varphi$ is not bounded on $\sigma(|f|e)$, then $\|\varphi(|f|e)\| = 
\sup\{\varphi(\lambda) : \lambda \in \sigma(|f|e)\} = \infty$. We proceed to show that then 
$\varphi(\||f|e\|) = \infty$. If now $\sigma(|f|e)$ was an unbounded subset of $[0, \infty)$, 
we would already have $\||f|e\| = \infty$, and hence $\varphi(\||f|e\|) = \infty$ as required. 
Thus let $\sigma(|f|e)$ be a bounded subset of $[0, \infty)$. As noted previously, this forces 
$\||f|e\| = \mathrm{max}\{\lambda : \lambda \in \sigma(|f|e)\}$. Since $\varphi$ is increasing 
on $[0, \infty]$ with $\varphi(0) = 0$, we must have $\varphi(\||f|e\|) \geq \varphi(\lambda) 
\geq 0$ for any $\lambda \in \sigma(|f|e)$. The fact that $\varphi$ is unbounded on 
$\sigma(|f|e)$ therefore forces $\varphi(\||f|e\|) = \infty$ as required.
The rest follows from (cf \cite[2.3(1), and 2.5(iv)]{FK}): $\mu_t(g) = 
\inf\{ ||ge||; e \in \M_0 \ {\rm with} \ \tau(\I - e) \leq t \}$.
\end{proof}

\begin{proposition}\label{DPvsK} Let $\varphi$ be an Orlicz function and let $f \in \widetilde{\M}$ be given. There exists some 
$\alpha > 0$ so that $\int_0^\infty \varphi(\alpha\mu_t(|f|))\, \mathrm{d}t < \infty$ if and only if there exists 
$\beta > 0$ so that $\varphi(\beta|f|) \in \widetilde{\M}$ and $\tau(\varphi(\beta|f|)) < \infty$. Moreover $$\|\mu(f)\|_\varphi = \inf\{\lambda > 0 : \varphi\left(\frac{1}{\lambda}|f|\right) \in \widetilde{\M}, \tau\left(\varphi\left(\frac{1}{\lambda}|f|\right)\right) \leq 1\}.$$
\end{proposition}

\begin{proof} The validity of this result for the case $b_\varphi = \infty$, was noted in the discussion preceding the lemma. Hence let $b_\varphi < \infty$. If now there exists $\beta > 0$ so that $\varphi(\beta|f|) \in \widetilde{\M}$, then by the lemma $$\tau(\varphi(\beta|f|)) =  \int_0^\infty \varphi(\beta\mu_t(|f|))\, \mathrm{d}t.$$The ``only if'' part of the first claim therefore follows. To see the converse, suppose that $\int_0^\infty \varphi(\alpha\mu_t(|f|))\, \mathrm{d}t < \infty$ for some $\alpha > 0$. If for some $t_0 > 0$ we had $\alpha\mu_{t_0}(f) > b_\varphi$, then of course $\alpha\mu_{t}(f) \geq \alpha\mu_{t_0}(f) > b_\varphi$ for all $0 \leq t \leq t_0$, which would force $$\int_0^\infty \varphi(\alpha\mu_t(|f|))\, \mathrm{d}t \geq \int_0^{t_0} \varphi(\alpha\mu_t(|f|))\, \mathrm{d}t = \int_0^{t_0} \infty\, \mathrm{d}t = \infty.$$Thus we must have $\alpha\mu_{t}(f) \leq b_\varphi$ for all $0 < t$. Since $t \to \mu_t(f)$ is right-continuous, this means that $\alpha\|f\|_\infty = \lim_{t\to 0^+}\alpha\mu_t(f) \leq b_\varphi < \infty$. So in this case we clearly have that $f \in \M$, with $\varphi(\frac{\alpha}{1+\epsilon} f) \in \M \subset \widetilde{\M}$. On applying the lemma we conclude that 
\begin{equation}\label{eq:1}
\tau\left(\varphi\left(\frac{\alpha}{1+\epsilon} f\right)\right) = \int_0^\infty \varphi\left(\frac{\alpha}{1+\epsilon}\mu_t(|f|)\right)\, \mathrm{d}t \leq \int_0^\infty \varphi(\alpha\mu_t(|f|))\, \mathrm{d}t < \infty
\end{equation}
as required.

To see the second claim, observe that the lemma ensures that $$\{\lambda > 0 : \varphi\left(\frac{1}{\lambda}|f|\right) \in \widetilde{\M}, \tau\left(\varphi\left(\frac{1}{\lambda}|f|\right)\right) \leq 1\} \subset \{\lambda > 0 : \int_0^\infty \varphi\left(\frac{1}{\lambda}\mu_t(|f|)\right)\, \mathrm{d}t \leq 1\}.$$Hence 
\begin{eqnarray*}
\|\mu(f)\|_\varphi &=& \inf\{\lambda > 0 : \int_0^\infty \varphi\left(\frac{1}{\lambda}\mu_t(|f|)\right)\, \mathrm{d}t \leq 1\}\\
&\leq& \inf\{\lambda > 0 : \varphi\left(\frac{1}{\lambda}|f|\right) \in \widetilde{\M}, \tau\left(\varphi\left(\frac{1}{\lambda}|f|\right)\right) \leq 1\}.
\end{eqnarray*}
To see that equality holds, let $\epsilon > 0$ be given, and select $\lambda_0 > 0$ so that $$\|\mu(f)\|_\varphi \leq \lambda_0 \leq (1+\epsilon)\|\mu(f)\|_\varphi \quad\mbox{and}\quad \int_0^\infty \varphi\left(\frac{1}{\lambda_0}\mu_t(|f|)\right)\, \mathrm{d}t \leq 1.$$But then by formula \ref{eq:1}, we have that $\varphi(\frac{1}{(1+\epsilon)\lambda_0} f) \in \widetilde{\M}$, with
$$\tau\left(\varphi\left(\frac{1}{(1+\epsilon)\lambda_0} f\right)\right) \leq \int_0^\infty \varphi\left(\frac{1}{\lambda_0}\mu_t(|f|)\right)\, \mathrm{d}t \leq 1.$$
So 
$$\inf\{\lambda > 0 : \varphi\left(\frac{1}{\lambda}|f|\right) \in \widetilde{\M}, \tau\left(\varphi\left(\frac{1}{\lambda}|f|\right)\right) < 1\} \leq (1+\epsilon)\lambda_0 \leq (1+\epsilon)^2\|\mu(f)\|_\varphi.$$
Since $\epsilon > 0$ was arbitrary, we have $$\inf\{\lambda > 0 : \varphi\left(\frac{1}{\lambda}|f|\right) \in \widetilde{\M}, \tau\left(\varphi\left(\frac{1}{\lambda}|f|\right)\right) \leq 1\} \leq  \|\mu(f)\|_\varphi$$as required.
\end{proof}

We close this section by formulating one more fact regarding Orlicz spaces that will prove to be useful later 
on. (This is a special case of known results in \cite{DDdP3}).

\begin{proposition}\label{kothe}
Let $\varphi$ be an Orlicz function and $\varphi^*$ its complementary function. Then $L^{\varphi^*}(\widetilde{\M})$ equipped with the norm $\|\cdot\|^0_\varphi$ defined by $$\|f\|^0_\varphi = \sup \{\tau(|fg|) : g \in L^{\varphi}(\widetilde{\M}), \|g\|_\varphi \leq 1\} \qquad f \in L^{\varphi^*}(\widetilde{\M})$$is the K\"{o}the dual of $L^{\varphi^*}(\widetilde{\M})$. That is $$L^{\varphi^*}(\widetilde{\M}) = \{f \in \widetilde{\M} : fg \in L^1(\M, \tau) \text{ for all } g \in L^{\varphi}(\widetilde{\M})\}.$$Consequently
$$|\tau(fg)| \leq \|f\|^0_\varphi.\|g\|_\varphi \quad\mbox{for all}\quad f \in L^{\varphi^*}(\widetilde{\M}), g \in L^{\varphi}(\widetilde{\M}).$$
\end{proposition}

\begin{proof}
It is clear from the discussion following Corollary 2.7 of \cite{DDdP2} that $L^{\varphi}(\widetilde{\M})$ is fully symmetric in the sense defined there. But by \cite[Corollary 2.6]{DDdP2}, $L^{\varphi}(\widetilde{\M})$ is then properly symmetric in the sense defined on p 737 of \cite{DDdP3}. The claims therefore follow from \cite[Corollary 4.8.15]{BS} and \cite[Theorem 5.6]{DDdP3}.
\end{proof}

\section{Noncommutative regular statistical models}

We begin with the definition of the classical regular model (cf \cite{PS}). Let $\{\Omega, \Sigma, \nu\}$
be a measure space; $\nu$ will be called the reference measure. The set of  densities of all the probability measures equivalent to $\nu$ will be called the state space $\cS_{\nu}$, i.e.
\begin{equation}
{\cS}_{\nu} = \{ f  \in L^1(\nu): f>0 \quad \nu-a.s.,\, E_1(f) =1 \},
\end{equation}
where, in general, $E_g(f) \equiv \int f \cdot g d\nu$.
\begin{definition}
\label{clasmodel}
The classical statistical model consists of the measure space $\{\Omega, \Sigma, \nu\}$, state space $\cS_{\nu}$, and the set of measurable functions $L^0(\Omega, \Sigma, \nu)$.
\end{definition}
As a next step, we wish to select regular random variables, i.e.  random variables having all finite moments. To this end
we define the moment generating functions as follows: fix $f \in {\cS}_{\nu}$, take a real random variable $u$ on $(\Omega, \Sigma, f d\nu)$
and define:
\begin{equation}
\hat{u}_f(t) = \int exp(tu) f d\nu, \qquad t \in \bR. 
\end{equation}

 In the sequel we will need the following properties of $\hat{u}$ ({\it for details see Widder, \cite{Wid}}):
\begin{enumerate}
\item $\hat{u}$ is analytic in the interior of its domain,
\item its derivatives are obtained by differentiating under the integral sign.
\end{enumerate}
Now the following definition is clear:
\begin{definition}
The set of all random variables such that
\begin{enumerate}
\item{} $\hat{u}_f$ is well defined in a neighborhood of the origin $0$,
\item{} the expectation of $u$ is zero,
\end{enumerate}
will be denoted by $L_f$ and called the regular random variables.
\end{definition}

We emphasize that {\it all the moments of every $u \in L_f$ exist and they are the values at $0$ of the derivatives of $\hat{u}_f$.}
In other words we have selected all random variables such that for each $u$, any moment $E_f(u^n)$ is well defined. The set of  regular random variables having zero expectation is characterized by:

\begin{theorem}(Pistone-Sempi, \cite{PS})
$L_f$ is the closed subspace of the Orlicz space $L^{\cosh - 1}(f\cdot \nu)$ of zero expectation random
variables.
\end{theorem}
Before proceeding with the noncommutative generalization of regular random variables, we want to make two remarks. Firstly, 
 the above result says that {\it classical Orlicz spaces are well motivated} in the context of probability calculus. Secondly, we kept the condition ``the expectation of $u$ is zero'' only to follow the original Pistone-Sempi argument.
 
Now, we turn to the noncommutative counterpart of the presented scheme. Let
 $({\M}, \tau)$ be a pair consisting of a semifinite von Neumann algebra and fns trace.
 Define (see \cite{Tak}, vol. I):
\begin{enumerate}
\item $n_{\tau} = \{ x \in {\M}: \tau(x^*x) < + \infty \}.$
\item ({\it definition ideal of the trace} $\tau$) $m_{\tau} = \{xy: x,y \in n_{\tau} \}.$
\item $\omega_x(y) = \tau(xy), \quad x\geq 0.$
\end{enumerate}

One has (for details see Takesaki, \cite{Tak}, vol. I)
\begin{enumerate}
\item if $x \in m_{\tau}$, and $x\geq 0$, then $\omega_x \in \M_*^+$. 
\item If $L^1(\M,\tau)$ stands for the completion of $(m_{\tau}, ||\cdot||_1)$ then $L^1(\M, \tau)$ is isometrically isomorphic to $\M_*$.
\item $\M_{*,0} \equiv \{ \omega_x : x \in m_{\tau} \}$ is norm dense in $\M_*$.
\end{enumerate}
Finally, denote by $\M_*^{+,1}$ ($\M_{*,0}^{+,1}$) the set of all normalized normal positive functionals in $\M_*$ (in $\M_{*,0}$ respectively).
Now, performing a ``quantization'' of Definition \ref{clasmodel} we arrive at
\begin{definition}
\label{qsm}
The noncommutative statistical model consists of a quantum measure space $(\M, \tau)$, ``quantum densities with 
respect to $\tau$'' in the form of $\M_{*,0}^{+,1}$, and
the set of $\tau$-measurable operators $\tM$.
\end{definition}

In the framework of the noncommutative statistical model the regular (noncommutative) random variables can be defined in the following way: 
\begin{definition}
\label{kwant}
\begin{equation}
L^{quant}_x = \{ g \in \tM: \quad 0 \in D(\widehat{\mu_{x}^g(t)})^0, \quad x \in m_{\tau}^+ \},
\end{equation}
where $D(\cdot)^0$ stands for the interior of the domain $D(\cdot)$ and
\begin{equation} 
\widehat{\mu_{x}^g(t)} = \int \exp(t\mu_s(g)) \mu_s(x)ds, \qquad t\in\bR.
\end{equation}
(Notice that the requirement that $0 \in D(\widehat{\mu_{x}^g(t)})^0$, presupposes that the transform $\widehat{\mu_{x}^g(t)}$ is well-defined in a neighbourhood of the origin.)
\end{definition}

We remind that above and in the sequel $\mu(g)$ ($\mu(x)$)  stands for the function $[0,\infty) \ni t \mapsto \mu_t(g) \in [0, \infty]$ 
($[0,\infty) \ni t \mapsto \mu_t(x) \in [0, \infty]$ respectively).

To comment on Definition \ref{kwant}, we should firstly clarify the role of $\mu(x)$. To this end we note that for $y \in \M$, $y\geq 0$, and $x \in m_{\tau},  x\geq 0$, one has
\begin{equation}
\label{la}
\omega_x(y) \equiv \tau(xy) = \int_0^{\infty}\mu_t(xy)dt \leq \int_0^{\infty} \mu_{\frac{t}{2}}(x)\mu_{\frac{t}{2}}(y)dt
\end{equation}
 and
 \begin{equation}
 \label{lala}
\omega_x(y^n) \equiv \tau(xy^n) \leq \int_0^{\infty}\mu_{\frac{t}{2}}(y^n) 
\mu_{\frac{t}{2}}(x)dt
\leq \int_0^{\infty} \mu_{\frac{t}{2n}}(y)^n \mu_{\frac{t}{2}}(x)dt.
\end{equation}
where we have used the Fack-Kosaki results \cite{FK} (Lemma 2.5) on generalized singular values.

Moreover, $[0, \infty) \ni t \mapsto \mu_t(x)$ is a positive function such that $\mu_t(x)\longrightarrow||x||$ as 
${t\downarrow 0}$, and for $x \in m_{\tau}^+$ corresponding to a state, one has
 \begin{equation}
\int_0^{\infty}\mu_t(x)dt = \tau(x) \equiv \tau(x {\I}) \equiv \omega_x({\I}) = 1.
\end{equation}
Therefore, the function $[0, \infty) \ni t \mapsto \mu_t(x) \in [0, \infty)$ plays the role of a density of a probability measure.

Secondly, let us turn to the role of the Laplace transform.
 It is an easy observation that the properties of the Laplace transform offers the existence of
$\int \mu_t(y)^n \mu_t(x)dt$ for any $n \in \bN$. Further, Fack-Kosaki results \cite{FK},  Lemma 2.5, lead to
\begin{eqnarray*}
\omega_x(y^n) &=& \tau(x y^n)\leq \tau(|xy^n|)\\
&=& \int_0^{\infty} \mu_t(x y^n)dt\\
&\leq& \int_0^{\infty} \mu_{t/2}(y^n)\mu_{t/2}(x)dt\\
&\leq& \int_0^{\infty}\mu_{t/2n}(y)^n \mu_{t/2}(x)dt \\
&=& 2n \int_0^{\infty} \mu_s(y)^n \mu_{ns}(x) ds\\
&\leq& 2n \int_0^{\infty} \mu_s(y)^n \mu_s(x) ds \\
&<& \infty.
\end{eqnarray*} 
But this gives the existence of moments of a noncommutative random variable $y$. 

\begin{definition}
Let $x \in L_+^1(\M, \tau)$ and let $\rho$ be a Banach function norm on $L^0((0, \infty), \mu_t(x)dt)$. In the spirit of \cite{DDdP} we then formally define the weighted noncommutative Banach function space $L^{\rho}_{x}(\tM)$ to be the collection of all $f \in \tM$ for which $\mu(f)$ belongs to $L^{\rho}((0, \infty), \mu_t(x)dt)$. For any such $f$ we write $\|f\|_\rho = \rho(\mu(f))$.
\end{definition}

\begin{theorem}\label{ncbf}
Let $x \in L_+^1(\M, \tau)$. If $\rho$ is a rearrangement-invariant Banach function norm on $L^0((0, \infty), \mu_t(x)dt)$ which satisfies the Fatou property, then $L^{\rho}_{x}(\tM)$ is a linear space and $\|\cdot\|_\rho$ a norm. Equipped with the norm $\|\cdot\|_\rho$, $L^{\rho}_{x}(\tM)$ is a Banach space which injects continuously into $\tM$.
\end{theorem}

\begin{proof}
We will not give a detailed proof, but only indicate how the argument in Section 4 of \cite{DDdP} may be adapted to the present context. For the sake of convenience, we will assume that $\tau(x) = 1$.

Since $t \to \mu_t(x)$ is decreasing, right-continuous on $[0, \infty)$, and finite-valued on $(0, \infty)$, it is actually Riemann-integrable on any bounded sub-interval of $(0, \infty)$, and zero-valued on $[t_x, \infty)$ where $t_x = \inf\{t > 0 : \mu_t(x) = 0\}$. These facts enable us to conclude that the function $$F_x(t) = \int_0^t \mu_s(x)ds \qquad t \geq 0$$is continuous and strictly increasing on $[0, t_x)$, and constant on $[t_x, \infty)$. So $F_x$ is actually a homeomorphism from $[0, t_x)$ onto $[0, 1)$. For any measurable function $g:[0, \infty) \to \mathbb{R}$ and any $t > 0$ we therefore have $$\int_0^{F_x(t)}g(s)ds = \int_0^t g(F_x(s))\mu_s(x)ds$$by the change of variables formula (see for example p 155 of \cite{Tay}).

For ease of notation we write $\nu_x$ for the Borel measure $$\nu_x(E) = \int\chi_E(t)\mu_t(x)dt.$$ Since $\mu(x)$ is non-zero on $[0, t_x)$, it is a simple matter to conclude that on $[0, t_x)$, $\nu_x \ll \lambda$ and $\lambda \ll \nu_x$ (here $\lambda$ denotes Lebesgue measure). Since $\nu_x$ is a finite measure, it is in fact $\epsilon-\delta$ absolutely continuous with respect to $\lambda$. Using these facts, one is now able to show that $\nu_x$ is non-atomic. (If $\nu_x(E) \neq 0$ then so too $\lambda(E) \neq 0$. Now set $\epsilon = \frac{1}{2}\nu_x(E)$, and select $F \subset E \cap [0, t_x)$ with $0 < \lambda(F) \leq \delta(\epsilon)$ to see that $\nu_x$ is non-atomic.) Thus by \cite[Theorem 2.2.7]{BS}, $\nu_x$ is a resonant measure.

In view of the fact that $\mu(x)$ is decreasing, we have that $F_x(t) = \nu_x([0, t]) \geq \nu_x([s, s+t])$ for any $s,t > 0$. More 
generally by approximating with intervals, one can show that for any $t > 0$ and any Borel set $E$ in $[0, \infty)$ 
with $\lambda(E) = t$, we have that $\nu_x(E) \leq \nu_x([0, t]) = F_x(t)$. Given some measurable function $f$ on $[0, 
\infty)$, these facts ensure that $$\inf\{\|f\chi_E\|_\infty : \lambda(E^c) \leq t\} \geq \inf\{\|f\chi_E\|_\infty : 
\nu_x(E^c) \leq F_x(t)\}.$$In other words $$\widetilde{\mu}_t(f, \lambda) \geq \widetilde{\mu}_{F_x(t)}(f, \nu_x).$$(The centered expressions above respectively denote the decreasing rearrangement of $f$ computed using $\lambda$ and $\nu_x$.)

Now if $h$ is decreasing and right-continuous on $[0, \infty)$, and finite valued on $(0, \infty)$, then more can be said. It is an exercise to see that in this case $\inf\{\|h\chi_E\|_\infty : \lambda(E^c) \leq t\} = \|h\chi_{(t, \infty)}\|_\infty$. (To see that ``$\leq$'' holds is trivial. For the converse note that if $\lambda(E \cap [0, t]) \neq 0$, then $\|h\chi_E\|_\infty \geq \|h\chi_{(t, \infty)}\|_\infty$ by the fact that $h$ is decreasing.) The right-continuity of $h$ combined with the fact that it is decreasing, ensures that $\|h\chi_{(t, \infty)}\|_\infty = h(t)$. A similar argument to the above shows that for any $0 < t < t_x$, we have that $\inf\{\|h\chi_E\|_\infty : \nu_x(E^c) \leq F_x(t)\} = \|h\chi_{(t, \infty)}\|_\infty$. For functions such as these, we therefore have $$h(t) = \widetilde{\mu}_t(h, \lambda) = \widetilde{\mu}_{F_x(t)}(h, \nu_x) \qquad \mbox{for all} \qquad 0 < t < t_x.$$

Finally let $a, b \in L_x^\rho(\tM)$ be given. We first show that then $a+b \in L_x^\rho(\tM)$, and hence 
that $L_x^\rho(\tM)$ is linear, before going on to conclude that $\|\cdot\|_\rho$ is a norm. By Theorem 3.4 of 
\cite{DDdP} we have that $$\int_0^t \widetilde{\mu}_s(\mu(a+b)-\mu(b), \lambda)ds \leq \int_0^t \widetilde{\mu}_s(\mu(a), \lambda)ds = \int_0^t \mu_t(a)ds$$for any $t > 0$. If now we apply Hardy's Lemma \cite[Proposition 2.3.6]{BS} to the decreasing function $\mu(x)\chi_{[0, t]}$, we may conclude that $$\int_0^t \widetilde{\mu}_s(\mu(a+b)-\mu(b), \lambda)\mu_s(x)ds \leq \int_0^t \mu_s(a)\mu_s(x)ds$$for any $t > 0$. Next use the facts that $\mu_s(a) = \widetilde{\mu}_{F_x(s)}(\mu(a), \nu_x)$ for all $0 < s < t_x$, and $\widetilde{\mu}_s(\mu(a+b)-\mu(b), \lambda) \geq \widetilde{\mu}_{F_x(s)}(\mu(a+b)-\mu(b), \nu_x)$, to get 
$$\int_0^t \widetilde{\mu}_{F_x(s)}(\mu(a+b)-\mu(b), \nu_x)\mu_s(x)ds \leq \int_0^t \widetilde{\mu}_{F_x(s)}(\mu(a), \nu_x)\mu_s(x)ds$$for any $t_x \geq t > 0$. 
Since $F_x$ is a homeomorphism from $[0,t_x)$ to $[0, 1)$, the change of variables formula now ensures that $$\int_0^r \widetilde{\mu}_{s}(\mu(a+b)-\mu(b), \nu_x)ds \leq \int_0^r \widetilde{\mu}_{s}(\mu(a), \nu_x)ds$$for any $1 > r > 0$. (Simply let $F_x(t) = r$.) Since $\nu_x$ is a probability measure, we in fact have that $\widetilde{\mu}_{s}(\mu(a+b)-\mu(b), \nu_x) = \widetilde{\mu}_{s}(\mu(a), \nu_x) = 0$ for all $s \geq 1$. Hence the 
previous centered inequality actually holds for all $r > 0$. We may now finally apply \cite[Theorem 2.4.6]{BS} to 
conclude that $\rho(\mu(a+b)-\mu(b)) \leq \rho(\mu(a))$. But since $\mu(a), \mu(b) \in L^{\rho}((0, \infty), \mu_t(x)dt)$, this inequality surely forces $\mu(a+b) - \mu(b) \in L^{\rho}((0, \infty), \mu_t(x)dt)$, and hence $\mu(a+b) \in L^{\rho}((0, \infty), \mu_t(x)dt)$. Thus by definition $a+b \in L^{\rho}_x(\tM)$, ensuring that $L^{\rho}_x(\tM)$ is linear. (The fact that $\alpha a \in L^{\rho}_x(\tM)$ whenever $a \in L^{\rho}_x(\tM)$ is easy to verify.) But then 
this same inequality also ensures that $\|a+b\|_\rho \leq \|a\|_\rho + \|b\|_\rho$, and hence that $\|\cdot\|_\rho$ is a semi-norm on $L_x^\rho(\tM)$. Now observe that if $\|a\|_\rho = \rho(\mu(a)) = 0$, then $\mu(a) = 0$ $\nu_x$-ae. But since we have that $\lambda \ll \nu_x$ on $[0, t_x)$, this fact forces $\mu_t(a) = 0$ for $\lambda$-almost every $t$ in $[0, t_x)$. The right-continuity of $t \to \mu_t(a)$ then ensures that $\|a\| = \lim_{t \downarrow 0}\mu_t(a) = 0$, and hence that $a = 0$. Thus $\|\cdot\|_\rho$ is in fact a norm. 

The rest of the proof runs along similar lines as the argument in Section 4 of \cite{DDdP}.
\end{proof}

Finally we wish to prove a non-commutative version of the Pistone-Sempi theorem.
 
\begin{theorem}\label{QPS}
The set $L^{quant}_x$  coincides with the closed subspace of the weighted Orlicz space $L_x^{\cosh - 1}(\tM) \equiv L^{\psi}_{x}(\tM)$ (where $\psi = \cosh -1$) of noncommutative random variables with a fixed expectation.
\end{theorem}

\medskip

In the above, the noncommutative space $L^{\psi}_{x}(\tM)$, is the Banach Function space defined by $f \in L^{\psi}_{x}(\tM) \Leftrightarrow \mu(f) \in L^{\psi}((0, \infty), \mu_t(x)dt)$. This space is a quantization (in two steps) of the space $L^{\cosh - 1}(f \cdot \nu)$ in the Pistone-Sempi theorem, in the following sense: Proceeding from $x \in m_{\tau}$, we view the decreasing rearrangement $t \to \mu_t(x)$ of $x$ as some sort of density of a probability measure, and use this density to produce the classical weighted Orlicz space $L^{\psi}((0, \infty), \mu_t(x)dt)$. By \cite[Theorem 4.8.9]{BS}, the Luxemburg norm $\|\cdot\|_{\cosh-1}$ on this Orlicz space is a rearrangement invariant function norm satisfying the Fatou Property. These facts can also be seen directly. For example to see that the norm is rearrangement invariant, we may simply apply Proposition \ref{DPvsK} to the von Neumann algebra $L^\infty((0, \infty), \nu_x)$ equipped with the fns trace $f \to \int_0^\infty f(t)\mu_t(x)dt$. It is moreover easy to see that $0 \leq f_n \uparrow f$ ~~ $\mu_t(x)dt$-a.e if and only if $0 \leq (\cosh-1)(f_n) \uparrow (\cosh-1)(f)$ $\mu_t(x)dt$-a.e. The fact that then $\rho(f_n) \uparrow \rho(f)$, therefore follows from the usual monotone convergence theorem. Since this is a quantized version of a 
\emph{weighted} Orlicz space on $(0,\infty)$, results like Proposition \ref{DPvsK} do not apply. (The extent to which this space resembles spaces like $L^\psi(\tM)$, will be discussed at a later stage.)

The space $L^{\psi}_{x}(\tM)$ is then the noncommutative version of $L^{\psi}((0, \infty), \mu_t(x)dt)$, defined in the spirit of the prescription originally given in \cite{DDdP}. The fact that $L^{\psi}_{x}(\tM)$ is a concrete well-defined Banach space, follows from Theorem \ref{ncbf}. The primary difference between the space constructed here, and the version discussed in \cite{DDdP}, is that Lebesgue measure has been replaced with the measure $\mu_t(x)dt$.

\medskip

\begin{proof}[Proof of Theorem \ref{QPS}]
Assume $g \in L_{x}^{\psi}(\tM)$ with $\psi \equiv \cosh - 1$.
Then $\mu(g)$ belongs to $L^{\psi}((0, \infty), \mu_t(x)dt)$.
Hence, there exists $a>0$ such that 
\newline
$E_x(\frac{1}{2}(\exp(\frac{\mu(g)}{a}) +  \exp(\frac{-\mu(g)}{a})) - 1)) < \infty,$
where $E_x(\phi) \equiv \int_0^{\infty} \phi(t) \mu_t(x) dt$.
However, as $E_x(-1) = \int_0^{\infty} (-1) \mu_t(x))dt < \infty$ then
$E_x((\exp(\frac{\mu(g)}{a}) +  \exp(\frac{-\mu(g)}{a}))) < \infty$. But as $(\frac{-1}{a}, \frac{1}{a}) \ni t \mapsto e^{t\mu(g)}$ is convex then 
\begin{equation}
e^{\alpha (\frac{-1}{a}) \mu(g) + (1 - \alpha) \frac{1}{a} \mu(g)} \leq \alpha e^{\frac{-\mu(g)}{a}} + (1 - \alpha)e^{\frac{\mu(g)}{a}}.
\end{equation}
with $\alpha \in [0,1]$.
Thus
\begin{equation}
\int_0^{\infty} e^{s\mu(g)} \mu_t(x) dt < \infty,
\end{equation}
for $s \in (\frac{-1}{a}, \frac{1}{a})$.
Consequently, $g \in L^{quant}_{x}$.

Conversely, let $g \in L^{quant}_{x}$.
Then, there exists $s$ such that both $s$ and $-s$ are in the domain of $\widehat{\mu_t(g)}$. This means that 
\begin{equation}
E_{x}(e^{s\mu(g)} + e^{-s\mu(g)}) < \infty.
\end{equation}
Consequently
\begin{equation}
\mu_t(g) \in L^{\psi}(\bR^+, \mu_t(x)dt).
\end{equation}
but this means $g \in L_{x}^{\psi}(\tM)$
\end{proof}

Therefore, to get the noncommutative regular statistical model, in Definition \ref{qsm}, one should restrict $\tM$ to $\{ L^{quant}_{x} \}$. Furthermore, we have
   
\begin{corollary}
There exists a quantum analog of the Pistone-Sempi theorem. Moreover, non-commutative Orlicz spaces are as well motivated for a description of noncommutative regular statistical models as the classical Orlicz spaces for classical regular statistical models.
\end{corollary}

To fully clarify the role of $\mu_t(x)$, we end this Section with

\begin{proposition} For any $0 \neq x \in L^1_+(\widetilde{\M})$, the quantity 
$$\tau_x(f) = \int_0^\infty \mu_t(f)\mu_t(x)dt \quad f \in \M^+$$(used implicitly in Theorem \ref{QPS}) is 
\emph{almost} a normal finite faithful trace in the sense that
\begin{itemize}
\item $\tau_x$ is subadditive, positive-homogeneous, and satisfies $\tau_x(a^*a) = \tau_x(aa^*)$ for every $a \in \M$;
\item $\tau_x(\I) < \infty$, and for any $a \in \M^+$ the situation $\tau_x(a) = 0$ forces $a=0$;
\item $\sup_n \tau_x(f_n) = \tau_x(f)$ for every sequence $\{f_n\}$ in $\M^+$ increasing to some $f \in \M^+$.
\end{itemize}
\end{proposition}

\begin{proof} Note that by \cite[Lemma 2.5]{FK}, we have that $\mu_t(f^*f) = \mu_t(|f|)^2 = \mu_t(f)^2 = \mu_t(f^*)^2 = \mu_t(|f^*|)^2 = \mu_t(ff^*)$ and also that $\mu_t(\alpha f) = |\alpha|\mu_t(f)$ for 
each $t >0$. This is enough to ensure that $\tau_x$ is positive-homogeneous, and satisfies the trace property $\tau_x(a^*a) = \tau_x(aa^*)$. Next let $a, b \in \M^+$ be given. From the proof of \cite[Theorem 3.4]{DDdP} it is then clear that $\int_0^t |\mu_s(a+b)- \mu_s(a)|ds \leq \int_0^t \mu_{s}(b)ds$ for any $t > 0$ (simply apply what is proved there to the set $T = [0, t]$).  In view of the fact that $t \to \mu_t(x)$ is decreasing, we may then apply Hardy's Lemma \cite[Theorem 2.3.6]{BS}, to conclude that 
\begin{eqnarray*}
\tau_x(a+b) - \tau_x(a) &=& \int_0^\infty (\mu_t(a+b)- \mu_t(a))\mu_t(x)dt\\
&\leq& \int_0^\infty |\mu_{t}(a+b) - \mu_{t}(a)|\mu_t(x)dt\\
&\leq& \int_0^\infty \mu_{t}(b)\mu_{t}(x)dt\\
&=& \tau_x(b).
\end{eqnarray*}

Let $\alpha = \tau(\I)$. Using \cite[Lemma 2.6]{FK}, the fact $\tau_x$ is finite, is then a simple consequence 
of the observation that $\tau_x(\I) = \int_0^\alpha \mu_t(\I)\mu_t(x)dt = \int_0^\alpha \mu_t(x)dt = 
\tau(x) < \infty$. 

Given any $f \in \M^+$, it is clear that if $0 = \tau_x(f) = \int_0^\infty 
\mu_t(f)\mu_t(x)dt$, then $\mu_t(f)\mu_t(x) = 0$ for all $t > 0$. (Use the fact that $t \to 
\mu_t(f)\mu_t(x)$ is decreasing.) Since $t \to \mu_t(x)$ is decreasing, we may conclude from 
the inequality $0 < \tau(x) = \int_0^\infty \mu_t(x)dt$, that there exists some $\delta > 0$ so that 
$0 \neq \mu_t(x)$ for all $0 \leq t < \delta$. But then we must have $0 = \mu_t(f)$ for all $0 < t 
< \delta$. The fact that $t \to \mu_t(f)$ is decreasing, ensures that $\mu_t(f) = 0$ for all $t > 0$, and 
hence that $\|f\| = \lim_{t\to 0^+}\mu_t(f) = 0$. 

It remains to verify the claim about increasing sequences. To this end suppose that we are given a 
sequence $\{f_n\} \subset \M$ increasing to some $f \in \M$. Since $\mu_t(f_n) \leq \mu_t(f)$ for 
each $n$ and each $t$, it is a simple matter to conclude from this that $\limsup_n \tau_x(f_n) 
\leq \tau_x(f)$. On the other hand \cite[Proposition 1.7]{DDdP3} ensures that $\mu_t(f) = \liminf_n
\mu_t(f_n)$. By the usual Fatou's lemma, this in turn enables us to conclude that $\tau_x(f) = 
\int_0^\infty \liminf_n \mu_t(f_n)\mu_t(x)dt \leq \liminf_n \int_0^\infty \mu_t(f_n)\mu_t(x)dt 
= \liminf_n \tau_x(f_n)$. We then clearly have that $\tau_x(f) = \lim_n\tau_x(f_n) = \sup_n \tau_x(f_n)$. 
\end{proof}

In the case where $\tau_x$ does happen to be a normal trace, we may use this quantity to obtain an alternative description of the space $L^{\cosh-1}_x(\widetilde{\M})$.

\begin{remark}
Let $\tau_x$ be as before, and suppose that $\tau$ is finite, and $\tau_x$ a normal trace. Then the weighted noncommutative Banach Function space $ L^{\cosh-1}_x(\tM, \tau)$ agrees up to isometry with the 
noncommutative Orlicz space $L^{\cosh-1}(\tM, \tau_x)$. (Here we have deliberately modified our usual notational convention, to in each case clearly show which trace is being used in the construction of the particular noncommutative space.)

Since both $\tau_x$ and $\tau$ are finite, it is clear that any operator affiliated to $\M$ is both 
$\tau$-measurable and $\tau_x$-measurable (see \cite[Proposition I.21(vi)]{Tp}). Hence we will simply speak of 
measurable operators in the rest of this remark.  Moreover for the Orlicz function 
$\psi = \cosh - 1$, we have that $a_\psi = 0$ and $b_\psi = \infty$. From the discussion preceding Lemma \ref{DPvsKlemma}, it is clear that for any measurable element $f$, $\psi(|f|)$ will again be measurable. By definition (see \cite{DDdP}) such an $f$ belongs to the noncommutative Banach Function Space $L^{\cosh-1}_x(\widetilde{\M}, \tau)$ if and only if $t \to \mu_t(f, \tau)$ belongs to the Orlicz space $L^{\cosh-1}((0,\infty), \mu_t(x)dt)$. But then by Proposition \ref{DPvsK}
 
\begin{tabular}{rcl}
$f \in L^{\cosh-1}(\M, \tau_x)$ & $\Leftrightarrow$ & $\tau_x(\psi(\alpha|f|)) < \infty$ for some $\alpha > 0$\\
& $\Leftrightarrow$ & $\int_0^\infty \mu_t(\psi(\alpha|f|))\mu_t(x)dt < \infty$ for some $\alpha > 0$\\
& $\Leftrightarrow$ & $\int_0^\infty \psi(\alpha(\mu_t(f))\mu_t(x)dt < \infty$ for some $\alpha > 0$\\
& $\Leftrightarrow$ & $f \in L^{\cosh-1}_x(\widetilde{\M}, \tau)$, 
\end{tabular}

\noindent where as before $\cosh - 1 = \psi$. Equality of the norms follows from the fact that $\tau_x(\psi(\frac{1}{\lambda}|f|)) = 
\int_0^\infty \psi(\frac{1}{\lambda}(\mu_t(f))\mu_t(x)dt$ for each $\lambda > 0$.
\end{remark}

\section{Defining Composition Operators}\label{defcomp}

Before proceeding to a definition of composition operators on noncommutative spaces, 
we briefly revise the conceptual framework in the classical case. Given two topological 
vector spaces $F_i(X_i) (i=1,2)$ of functions defined 
on sets $X_1$ and $X_2$ respectively, a continuous linear operator 
$C: F_1(X_1) \rightarrow F_2(X_2)$ is called a composition operator if it is 
of the form 
            $$C(f) = f{\circ}T \quad f \in F_1(X_1)$$
for some transformation $T: X_2 \rightarrow X_1$. In the various contexts,
for example $C^{(p)}, L^p$ and $H^p$ spaces, the theory of composition operators 
relates to the study of differentiable, measurable, and analytic transformations 
respectively. For an introduction to the general theory of composition operators 
see for example \cite{SM}. For a workable theory in a given context one should 
firstly be able to distinguish those transformations which induce composition 
operators, and secondly be able to distinguish those bounded linear operators which 
are indeed composition operators. Subsequent to the definition of composition 
operators on noncommutative spaces, we will look at the first of these issues.

Given Banach function spaces $L^{\rho_i}(X_i, \Sigma_i, m_i)$ ($i = 1, 2$) of measurable 
functions on given measure spaces, a continuous linear operator
     $$C: \LP{\rho_1}{1} \rightarrow \LP{\rho_2}{2}$$ 
is called a (generalised) composition operator if for some $Y \in \Sigma_2$ 
and some measurable transformation $T: Y \rightarrow X_1$ (ie. $T^{-1}(E) \in 
\Sigma_2$ whenever $E \in \Sigma_1$) $C$ is of the form
\begin{displaymath} 
C(f)(x) = \left\{ \begin{array}{ll}
       f{\circ}T(x) & x \in Y \\
       0            & x \in X{\backslash}Y 
       \end{array} \right. \qquad f \in \LP{p}{1}.
\end{displaymath}
In this case we will write $C=C_T$.

In the context of von Neumann algebras, the noncommutative analogue of a nonsingular 
measurable transformation is that of a normal Jordan $*$-morphism $J : \M_1 \to \M_2$. The 
following fact will be a useful tool in our quest to define noncommutative composition 
operators:

\begin{proposition}\cite[4.7(i)]{L1}
Let $\M_1, \M_2$ be semifinite von Neumann algebras and $J : \M_1 \to \M_2$ a Jordan 
$*$-morphism. Then $J$ extends uniquely to a continuous Jordan $*$-morphism $\widetilde{J} 
: \widetilde{\M_1} \to \widetilde{\M_2}$ iff $\tau_2 \circ J$ is $\epsilon-\delta$ absolutely 
continuous with respect to $\tau_1$ on the projection lattice of $\M_1$ (ie. for any 
$\epsilon > 0$ there exists $\delta > 0$ such that for any projection $e \in \mathbb{P}(\M_1)$ we have 
$\tau_2(J(e)) < \epsilon$ whenever $\tau_1(e) < \delta$).  
\end{proposition}

\begin{definition}
Let $\M_1, \M_2$ be semifinite von Neumann algebras, and $L^{\rho_1}(\widetilde{\M_1})$ and 
$L^{\rho_2}(\widetilde{\M_2})$ two noncommutative symmetric Banach Function spaces. Let $J 
: \M_1 \to \M_2$ be a normal Jordan $*$-morphism, and $\tau_2 \circ J$ be $\epsilon-\delta$ 
absolutely continuous with respect to $\tau_1$ on the projection lattice $\mathbb{P}(\M_1)$. If the 
unique continuous extension $\widetilde{J} : \widetilde{\M_1} \to \widetilde{\M_2}$ maps 
$L^{\rho_1}(\widetilde{\M_1})$ into $L^{\rho_2}(\widetilde{\M_2})$, we call the induced linear 
map $L^{\rho_1}(\widetilde{\M_1}) \to L^{\rho_2}(\widetilde{\M_2})$ a composition operator from 
$L^{\rho_1}(\widetilde{\M_1})$ into $L^{\rho_2}(\widetilde{\M_2})$, and will occasionally 
denote it by $C_J$ (in deference to the commutative practice). 
\end{definition}

We close this section with the observation that since 
$L^{\rho_i}(\widetilde{\M_i})$ injects continuously into $\widetilde{\M_i}$ \cite[4.4]{DDdP}, 
the continuity of $\widetilde{J}$, coupled with the closed graph theorem ensures that $C_J$ is 
continuous.

\section{Positive maps that induce bounded maps on Orlicz spaces}

Our first theorem is a rather simple consequence of interpolation theory.

\begin{theorem}\label{posmap}
Let $\M_1, \M_2$ be semifinite von Neumann algebras equipped with fns traces $\tau_1$ and $\tau_2$ respectively, and let $T: \M_1 \to \M_2$ be a positive map satisfying $\tau_2\circ T \leq C\tau_1$ for some constant $C > 0$. Then for 
any fully symmetric Banach function space $L^\rho(0, \infty)$, the restriction of  $T$ to $\M_1 \cap L^1(\M_1, \tau_1)$ canonically extends to a bounded map from $L^\rho(\widetilde{\M}_1)$ to $L^\rho(\widetilde{\M}_2)$.
\end{theorem}

\begin{proof} Firstly note that Yeadon \cite{Y} showed that under the conditions of the hypothesis, the restriction of $T$ to $\M_1 \cap L^1(\M_1, \tau_1)$ canonically extends to a bounded map from $L^1(\M_1, \tau_1)$ to $L^1(\M_2, \tau_2)$. (For a more recent version of this result also valid for Haagerup $L^p$-spaces, the reader is referred to \cite{HJX}).

The result now follows from the fact that in the language of \cite{DDdP2}, the pair $(L^\rho(\widetilde{\M}_1), L^\rho(\widetilde{\M}_2))$ is an \emph{exact interpolation pair} for the pair of Banach couples $((\M_1, L^1(\M_1, \tau_1)), (\M_2, L^1(\M_2, \tau_2))$. To see this let $x \in L^\rho(\widetilde{\M}_1)$ and $y \in \M_2 + L^1(\M_2, \tau_2)$ be given with $y \prec \prec x$ (that is with $$\int_0^\alpha \mu_t(|y|)\, \mathrm{d}t \leq \int_0^\alpha\mu_t(|x|)\, \mathrm{d}t \quad \mbox{for all} \quad \alpha > 0.)$$Then from Corollary 2.6 of 
\cite{DDdP2}, it is clear that the following implications hold: $x \in L^\rho(\widetilde{\M}_1), y \prec\prec x  \Leftrightarrow \mu(x) \in L^\rho(0, \infty), \mu(y) \prec\prec \mu(x) \Rightarrow \mu(y) \in L^\rho(0, \infty) \Leftrightarrow y \in L^\rho(\widetilde{\M}_2)$. Thus the claim follows from Corollary 2.5 of \cite{DDdP2}.
\end{proof}

\begin{remark}
From the discussion following Corollary 2.7 of \cite{DDdP2}, it is clear that the above theorem applies in particular to noncommutative Orlicz spaces.
\end{remark}

By $T: \M \to \M\subseteq \B(\mathfrak{H})$ we denote a completely positive unital normal map. 
Recall that any CP map $T$ is of the form
\begin{equation}
\label{stinespring}
T(f) = W^* \pi(f) W
\end{equation}
where $\pi: \M \to \B(\mathfrak{L})$ is a $^*$-normal representation of $\M$ in $\B(\mathfrak{L})$, and $W : \mathfrak{H} \to \mathfrak{L}$ is a linear bounded operator. It is worth pointing out that when $T$ is {\it unital}, then $W$ is an isometry.  Following Arveson \cite{Ar}, we say that a completely positive map $T: \M \to \M$ is pure if, for every completely positive map $T^{\prime}: \M \to \M$, the property $T - T^{\prime}$ is a completely positive map implies that $T^{\prime}$ is a scalar multiple of $T$. It was shown by Arveson \cite{Ar} that a non-zero  pure CP map $T$ is of the form (\ref{stinespring}) with $\pi$ being an irreducible representation. If (as is the case here) 
$T$ is a \emph{normal} pure CP map, the irreducible representation $\pi$ of $\M$ on $\B(\mathfrak{L})$ will also be normal. But in this case $\pi(\M)$ will be both irreducible and weak*-closed, whence $\pi(\M) = \B(\mathfrak{L})$.

We close this section by indicating the applicability of Theorem \ref{posmap} to Jordan $*$-morphisms and pure CP maps. 

\begin{proposition}
Let $T:\M_1 \to \M_2$ be a positive normal map and let $\mathcal{B}$ be the weak*-closed subalgebra of $\M_2$ generated by $T(\M_1)$. Then in either of the following cases there exists an fns trace $\tau_T$ on $\M_1$ satisfying $$\tau_2 \circ T \leq \tau_T:$$
\begin{itemize}
\item $T$ is a Jordan $*$-morphism for which the restriction of $\tau_2$ to $T(\M_1)$ is semifinite.
\item $\M_2 = \B(\mathfrak{H})$, and $T$ is a unital pure CP map. 
\end{itemize}
\end{proposition}

\begin{proof}
We first consider the case where $T$ is a normal Jordan $*$-morphism. 
Now let $z$ be a central projection in $\mathcal{B}$ 
such that $a\mapsto zT(a)$ is a *-homomorphism and $a\mapsto (T(\I)-z)T(a)$ a *-antihomomorphism. If
$T(a)=0$ for some $a\in \M_1$, then $T(ab)=zT(a)T(b)+(T(\I)-z)T(b)T(a)=0$ and similarly $T(ba)=0$. Hence the 
kernel of $T$ is a two-sided ideal, which is even weak*-closed because of $T$'s normality. Thus there exists a central
projection $e$ such that $\ker(T)=e\M_1$ (see Proposition II.3.12 of \cite{Tak}). We now define $\tau_T$ on $\M_1$ by 
$$\tau_T(a) = \tau_1(ea) + \tau_2(T((\I-e)a)) \quad\mbox{for all}\quad a \in \M_1^+.$$The centrality of $e$ ensures that the restriction of $\tau_1$ to $e\M_1$ is an fns trace. What remains to be done is to show that $\tau_2(T((\I-e)\cdot))$ is an fns trace on $(\I-e)\M_1$. The faithfulness follows from the injectivity of $T$ on $(\I-e)\M_1$, whereas the normality is a consequence of the normality of both $T$ and $\tau_2$. To see that $\tau_2(T((\I-e)\cdot))$ is actually a trace, we note that for any $a \in (\I-e)\M_1$,
\begin{eqnarray*}
\tau_2(T(a^*a)) &=& \tau_2(zT(a^*a)) + \tau_2((T(\I)-z)T(a^*a))\\
&=& \tau_2(zT(a^*)T(a)) + \tau_2((T(\I)-z)T(a)T(a^*))\\
&=& \tau_2(zT(a)T(a^*)) + \tau_2((T(\I)-z)T(a^*)T(a))\\
&=& \tau_2(zT(aa^*)) + \tau_2((T(\I)-z)T(aa^*))\\
&=& \tau_2(T(aa^*)).
\end{eqnarray*}

\medskip

Now suppose that $\M_2 = \B(\mathfrak{H})$, and that $T$ is a normal unital pure CP map. From the discussion preceding this proposition it is clear that $T$ is of the form $$T(f) = W^* \pi(f) W$$where $\pi: \M \to \B(\mathfrak{L})$ is a normal $*$-homomorphism onto some $\B(\mathfrak{L})$, and $W : \mathfrak{H} \to \mathfrak{L}$ is an isometric 
injection. In the present context $\tau_T$ is then defined by $\tau_T(a) = \tau_1(ea) + \mathrm{Tr}_\mathfrak{L}(\pi((\I-e)a)) \quad\mbox{for all}\quad a \in \M_1^+$, where $e$ is the central
projection $e$ for which $\ker(T)=e\M_1$.
If therefore we can show that $$\mathrm{Tr}_\mathfrak{H}(W^*\cdot W) \leq \mathrm{Tr}_\mathfrak{L}(\cdot),$$the conclusion will follow from the case considered above. To this end let $\{x_\nu\}$ be an ONB for $\mathfrak{H}$. Using the fact that $W^*W = \I_\mathfrak{H}$, it is now an easy exercise to show that $\{W(x_\nu)\}$ is an ONS in 
$\mathfrak{K} \equiv W(\mathfrak{H}) \subseteq \mathfrak{L}$. Hence for any $a \in \B(\mathfrak{L})^+$, 
$$\mathrm{Tr}_\mathfrak{H}(W^*aW) = \sum_\nu \langle W^*aW(x_\nu) , x_\nu\rangle = \sum_\nu \langle aW(x_\nu) , W(x_\nu)\rangle \leq \mathrm{Tr}_\mathfrak{L}(a).$$
\end{proof}

\section{Describing Jordan $*$-morphisms which induce composition operators}

Given an Orlicz function $\varphi$ and a projection $e \in \M$ with $0 < \tau(e) < \infty$, $e$ will then belong to $L^\varphi(\widetilde{\M})$, and the Luxemburg-Nakano norm of $e$ will be $$\|e\|_\varphi =
 \frac{1}{\varphi^{-1}(1/\tau(e))}.$$To see this note that for any $\alpha>0$, $\varphi(\alpha e) = 
\varphi(\alpha)e$. Since for any $0 < \alpha < b_\varphi$ we then have that $$\tau(\varphi(\alpha e)) = 
\varphi(\alpha)\tau(e) < \infty,$$ it is clear that $e \in L^\varphi(\widetilde{\M})$. For the claim 
regarding the norm estimate, we may use Proposition \ref{DPvsK} to see that 
\begin{eqnarray*}
\|e\|_\varphi &=& \inf\{\lambda > 0 : \varphi\left(\frac{1}{\lambda}\right)\tau(e) \leq 1\}\\
&=& \inf\left\{\lambda > 0 : \varphi\left(\frac{1}{\lambda}\right) \leq \frac{1}{\tau(e)}\right\}\\
&=& \left[\sup\left\{\nu > 0 : \varphi(\nu) \leq \frac{1}{\tau(e)}\right\}\right]^{-1}\\
&=& \frac{1}{\varphi^{-1}(1/\tau(e))}.
\end{eqnarray*}

Throughout this section $\M_1, \M_2$ will denote semifinite von Neumann algebras respectively equipped with fns traces $\tau_1, \tau_2$. In addition $J$ will denote a normal Jordan $*$-morphism $J:\M_1 \to \M_2$ for which $\tau_2 \circ J$ is a semifinite weight on $\M_1$. (Since the modular automorphism group of $\tau_1$ is trivial, this ensures the existence of the Radon-Nikodym derivative $\frac{d \tau_2\circ J}{d \tau_1}$ as a positive operator affiliated to $\M_1$ \cite[Theorem 5.12]{PT}.)  When studying those Jordan morphisms which for a pair of Orlicz functions $\varphi_1, \varphi_2$ induce bounded linear maps from $L^{\varphi_1}(\widetilde{\M_1})$ to $L^{\varphi_2}(\widetilde{\M_2})$, this restriction is entirely reasonable and natural. To see this suppose for example that $J$ was known to restrict to a $\|\cdot\|_{\varphi_1} - \|\cdot\|_{\varphi_2}$ continuous map from $\M_1\cap L^{\varphi_1}(\widetilde{\M_1})$ to $\M_2\cap L^{\varphi_2}(\widetilde{\M_2})$ with norm $K$. For any projection $e \in \M_1$ with $\tau_1(e) < \infty$ and $J(e) \neq 0$, we would then have $$0 < \frac{1}{{\varphi_2}^{-1}(1/\tau_2(J(e)))} = \|J(e)\|_{\varphi_2} \leq K\|e\|_{\varphi_1} < \infty.$$In other words $0 < {\varphi_2}^{-1}(1/\tau_2(J(e))) < \infty$. In the case where  $a_{\varphi_2} = 0, b_{\varphi_2} = \infty$, an application of $\varphi_2$ to this inequality would then yield $\tau_2(J(e) < \infty$. Thus in this case for any projection $e \in \M_1$ we would then have that $\tau_2(J(e) < \infty$ whenever $\tau_1(e) < \infty$. This is clearly sufficient to force the semifiniteness of $\tau_2 \circ J$. The previous centred equation can be reformulated as $$\frac{1}{{\varphi_2}^{-1}(1/\tau_2(J(e)))} \leq K\frac{1}{{\varphi_1}^{-1}(1/\tau_1(e))}.$$ In the case where $a_{\varphi_1} = a_{\varphi_2} = 0, b_{\varphi_1} = b_{\varphi_2} = \infty$, this inequality forces not just the semifiniteness of $\tau_2 \circ J$, but even ensures that $\tau_2 \circ J$ is $\epsilon - \delta$ absolutely continuous with respect to $\tau_1$.

When restricting attention to Jordan $*$-morphisms, the additional structure we have to work with in this case, 
enables us to significantly sharpen the results of the previous section for this class of maps. Our goal here is to 
actually characterise those normal Jordan $*$-morphisms $J : \M_1 \to \M_2$ which 
for a given pair of ``well-behaved'' noncommutative Orlicz spaces $L^{\varphi_1}(\widetilde{\M_1})$ and 
$L^{\varphi_2}(\widetilde{\M_2})$, yield composition operators from $L^{\varphi_1}(\widetilde{\M_1})$ 
to $L^{\varphi_2}(\widetilde{\M_2})$. We point out that in the case where $\varphi_1 \neq \varphi_2$, these 
results are new, even for classical Orlicz spaces! In the case $\varphi_1 = \varphi_2$, the constraints 
on our main theorem, are exactly the same as the results in the literature. (Compare the second part of Theorem 2.2 of 
\cite{CHKM} with the main theorem of this section.)

\begin{theorem}
\label{5.1}
Let $\psi, \varphi_1, \varphi_2$ be Orlicz functions for which $\psi\circ\varphi_2 = \varphi_1$, and let $J : \M_1 \to \M_2$ be a normal Jordan $*$-morphism for which $\tau_2 \circ J$ is semifinite on $\M_1$, and $\epsilon - \delta$ absolutely continuous with respect to $\tau_1$.

Consider the following claims:
\begin{enumerate}
\item $f_J = \frac{d \tau_2\circ J}{d\tau_1} \in L^{\psi^*}(\widetilde{\M_1})$;
\item the canonical extension of $J$ to a Jordan $*$-morphism from $\widetilde{\M_1}$ to $\widetilde{\M_2}$, restricts to a bounded map $C_J$ from $L^{\varphi_1}(\widetilde{\M_1})$ to $L^{\varphi_2}(\widetilde{\M_2})$.
\end{enumerate}
The implication $(1) \Rightarrow (2)$ holds in general. If $\varphi_2$ satisfies $\Delta_2$ for all $t$, the two statements are equivalent. If $(1)$ does hold, then the norm of $C_J$ restricted to the self-adjoint portion of $L^{\varphi_1}(\widetilde{\M_1})$, is majorised by $\mathrm{max}\{1, \|f_J\|^0_{\psi^*}\}$.
\end{theorem}

Before proceeding with the proof of this theorem, we pause to make a number of technical observations. Most of these are non-commutative versions of known facts about Orlicz functions.

\begin{lemma}\label{lemma2}
Let $\varphi$ be an Orlicz function, and $\M$ a semifinite von Neumann algebra with fns trace $\tau$. 
\begin{enumerate}
\item If $a \in \widetilde{\M}$  with $\tau(\varphi(|a|)) < \infty$, then for any $\beta \leq 1$ we have that $\varphi(\beta|a|) \in \widetilde{\M}$ with $\tau(\varphi(\beta|a|)) \leq \tau(\beta\varphi(|a|))$. In particular given $a \in L^{\varphi}(\widetilde{\M})$ with  $\|a\|_{\varphi} < 1$, we have that $\varphi(|a|) \in \widetilde{\M}$ with $\tau(\varphi(|a|)) < 1$.
\item If $a_\varphi > 0$, then for any $a \in \M$ we have $a \in L^{\varphi}(\widetilde{\M})$ with $$a_\varphi\|a\|_\varphi \leq \|a\|_\infty.$$
\item If $b_\varphi < \infty$, then for any $a \in L^{\varphi}(\widetilde{\M})$, we have $a \in \M$ with $$b_\varphi\|a\|_\varphi \geq \|a\|_\infty.$$ 
\end{enumerate} 
\end{lemma}

\begin{proof}
\begin{enumerate}
\item Firstly let $a \in \widetilde{\M}$ be given with $\tau(\varphi(|a|)) < \infty$. If $b_\varphi = \infty$, it is trivial to see that then $\varphi(\beta|a|) \in \widetilde{\M}$ for any $0 \leq \beta$, since in that case the continuity of $\varphi$ on $[0, \infty]$ and the fact that $|a| \in \widetilde{\M}$, is enough to force this conclusion. If on the other hand $b_\varphi < \infty$, then we must have that $b_\varphi \geq \|a\|_\infty$. This may be seen by suitably modifying the first part of the proof of Proposition \ref{DPvsK}. Specifically by Lemma \ref{DPvsKlemma} we will have $$\tau(\varphi(|a|)) =  \int_0^\infty \varphi(\mu_t(|a|))\, \mathrm{d}t < \infty.$$ If for some $t_0 > 0$ we had $\mu_{t_0}(a) > b_\varphi$, then of course $\mu_{t}(a) \geq \mu_{t_0}(a) > b_\varphi$ for all $0 \leq t \leq t_0$, which would force $$\int_0^\infty \varphi(\mu_t(|a|))\, \mathrm{d}t \geq \int_0^{t_0} \varphi(\mu_t(|a|))\, \mathrm{d}t = \int_0^{t_0} \infty\, \mathrm{d}t = \infty.$$
Thus we must have $\mu_{t}(a) \leq b_\varphi$ for all $0 < t$. Since $t \to \mu_t(f)$ is right-continuous, this means that $\|a\|_\infty = \lim_{t\to 0^+}\mu_t(a) \leq b_\varphi < \infty$. But then $\beta\|a\|_\infty < b_\varphi$ for any $\beta < 1$. This means in particular that $\sigma(\beta|a|)$ is contained in  $[0, b_\varphi)$. The continuity of $\varphi$ on $[0, b_\varphi]$, then ensures that $\varphi(\beta|a|) \in \M$. Thus in either case, $\varphi(\beta|a|) \in \widetilde{\M}$ whenever $0 \leq \beta < 1$. If now we combine the convexity of $\varphi$ with the fact that $\varphi(0) = 0$, we see that $\varphi(\beta t) \leq \beta \varphi(t)$ for any $t \geq 0$ and any $\beta \leq 1$. Thus $\varphi(\beta|a|) \leq \beta \varphi(|a|)$. An application of the trace, now yields the conclusion that $$\tau(\varphi(\beta|a|)) \leq \beta \tau(\varphi(|a|)) < \tau(\varphi(|a|)).$$

Now suppose we are given $a \in L^{\varphi}(\widetilde{\M})$ with  $\|a\|_{\varphi} < 1$. From the formula for the Luxemburg-Nakano norm in Proposition \ref{DPvsK} it follows that there exists $\alpha > 1$ so that $\varphi(\alpha|a|) \in \widetilde{\M}$ with $\tau(\varphi(\alpha|a|)) \leq 1$. It then follows from what we have just proved that $\varphi(|a|) \in \widetilde{\M}$ with $$1 \geq \tau(\varphi(\alpha |a|)) > \tau(\varphi(|a|)).$$

\item Let $a_\varphi > 0$ and suppose that we are given $b \in \M$ with $\|b\|_\infty = 1$. In view of the fact that $\varphi$ vanishes on $[0, a_\varphi]$ and that $\sigma(a_\varphi |b|) \subset [0, a_\varphi]$, we have that $\varphi(a_\varphi |b|) = 0$. We may now conclude from the formula for the Luxemburg-Nakano norm in Proposition \ref{DPvsK}, that $\frac{1}{a_\varphi} \geq \|b\|_\varphi$. The claim follows on replacing $b$ with $\frac{1}{\|a\|_\infty}a$.

\item Let $b_\varphi < \infty$. Given $\epsilon > 0$ and $0 \neq a \in L^{\varphi}(\widetilde{\M})$, select $\|a\|_\varphi \leq \alpha < \|a\|_\varphi + \epsilon$ so that $\varphi(\frac{1}{\alpha}|a|) \in \widetilde{\M}$ with $\tau(\varphi(\frac{1}{\alpha}|a|)) \leq 1$. Now recall that in the proof of claim (1), we showed that when $b_\varphi < \infty$, then for any $b \in \widetilde{\M}$ with $\tau(\varphi(|b|)) < \infty$, we will have $b_\varphi \geq \|b\|_\infty$. Applying this fact to $\frac{1}{\alpha}|a|$, yields the conclusion that $b_\varphi(\|a\|_\varphi + \epsilon) > b_\varphi\alpha \geq \|a\|_\infty$.
\end{enumerate} 
\end{proof}

The following fact is a simple consequence of the above lemma.

\begin{lemma}\label{lemma5}
Let $\varphi$ be an Orlicz function which satisfies $\Delta_2$ for all $t$, and let $\M$ be a semifinite von Neumann algebra with fns trace $\tau$. Then $a \in L^{\varphi}(\widetilde{\M})$ if and only if $\varphi(|a|) \in \widetilde{\M}$ and $\tau(\varphi(|a|)) < \infty$.
\end{lemma}

\begin{proof}
The converse being trivial, assume that $a \in L^{\varphi}(\widetilde{\M})$. If $\|a\|_\varphi < 1$, we are done by Lemma 
\ref{lemma2}. If $\|a\|_\varphi \geq 1$, we may select $\alpha < 1$ so that $\|\alpha|a|\|_\varphi < 1$. Thus $\tau(\varphi(\alpha|a|)) < 1$ by Lemma \ref{lemma2}. Since $\varphi$ satisfies $\Delta_2$ for all $t$, there exists a constant $K > 0$ so that $\varphi(t) \leq K\varphi(\alpha t)$. Hence by the Borel functional calculus $\varphi(|a|) \leq K\varphi(\alpha|a|)$. Both $K\varphi(\alpha|a|))$ and $\varphi(|a|)$ are affiliated to the commutative von Neumann algebra generated by the spectral projections of $|a|$. Given $\epsilon > 0$, the $\tau$-measurability of $\varphi(\alpha|a|)$ ensures that we may select a projection $e$ in this algebra with $\tau(\I - e) < \epsilon$ and $\varphi(\alpha|a|)e \in \M$. But then since $0 \leq \varphi(|a|)e \leq K\varphi(\alpha|a|)e$, we must have that $\varphi(|a|)e \in \M$ as well. Hence $\varphi(|a|) \in \widetilde{\M}$ with in addition $\tau(\varphi(|a|)) \leq K\tau(\varphi(\alpha|a|)) < \infty$.
\end{proof}

\begin{lemma}\label{lemma1}
As before let $\psi, \varphi_1, \varphi_2$ be Orlicz functions for which $\psi\circ\varphi_2 = \varphi_1$. Then $a_{\varphi_1} \geq a_{\varphi_2}$ and $b_{\varphi_1} \leq b_{\varphi_2}$. 
\end{lemma}

\begin{proof}
In view of the equality $\psi\circ\varphi_2 = \varphi_1$, we have that $\varphi_1(t) = 0$ whenever $\varphi_2(t) = 0$, and also that $\varphi_1(t) = \infty$ whenever $\varphi_2(t) = \infty$.
\end{proof}

\begin{remark}
Let $\psi, \varphi_1, \varphi_2$ be Orlicz functions for which $\psi\circ\varphi_2 = \varphi_1$, and let $J : \M_1 \to \M_2$ be a Jordan $*$-morphism. It is clear from the previous two lemmas that in the case where $0 < a_{\varphi_2} \leq b_{\varphi_2} < \infty$, we must also have $0 < a_{\varphi_1} \leq b_{\varphi_1} < \infty$, which in turn ensures that the spaces $L^{\varphi_1}(\widetilde{\M_1}), L^{\varphi_2}(\widetilde{\M_2})$ are just isomorphic copies of $\M_1$ and $\M_2$ respectively. Thus in this case $J$ of course trivially induces a ``composition operator'' from $L^{\varphi_1}(\widetilde{\M_1})$ to $L^{\varphi_2}(\widetilde{\M_2})$ with no further restrictions on $\varphi_1$ and $\varphi_2$.
\end{remark}

\begin{lemma}\label{lemma3}
Let $\psi, \varphi_1, \varphi_2$ be Orlicz functions for which $\psi\circ\varphi_2 = \varphi_1$, and let $\M$ be a semifinite von Neumann algebra with fns trace $\tau$.  For any $a \in L^{\varphi_1}(\widetilde{\M})$ with $\|a\|_{\varphi_1} < 1$, we have that $\varphi_2(|a|) \in L^{\psi}(\widetilde{\M})$ and $\|\varphi_2(|a|)\|_{\psi} \leq \|a\|_{\varphi_1}$.
\end{lemma}

\begin{proof}
Suppose that we are given $a \in L^{\varphi_1}(\widetilde{\M})$ and that for some $\alpha > 1$ we have that with $\varphi_1(\alpha|a|) \in \widetilde{\M}$, and $\tau(\varphi_1(\alpha|a|)) < \infty$. By Lemma \ref{DPvsKlemma} this last inequality can of course be written as 
$$\int_0^\infty \psi(\varphi_2(\alpha\mu_t(|a|)))\, \mathrm{d}t = \int_0^\infty \varphi_1(\alpha\mu_t(|a|))\, \mathrm{d}t < \infty.$$
Similar observations to those employed in the proof of Lemma \ref{lemma2}, suffice to show that the above integral cannot be finite if $\varphi_2(\alpha\mu_{t_0}(|a|)) = \infty$ for some $t_0 > 0$. But if $\varphi_2(\alpha\mu_t(|a|)) < \infty$ for every $t > 0$, then surely $\alpha\mu_t(|a|) \leq b_{\varphi_2}$ for every $t > 0$. By the right continuity of $t \to \mu_t(f)$, this means that $\alpha\|a\|_\infty = \lim_{t\to 0^+} \alpha\mu_t(|a|) \leq b_{\varphi_2}$. Now if $b_{\varphi_2} = \infty$, $\varphi_2$ is continuous on all of $[0, \infty]$, thus ensuring that $\varphi_2(|a|) \in \widetilde{\M}$. If $b_{\varphi_2} < \infty$, then by the  above inequality, we have that $\|a\|_\infty < \alpha\|a\|_\infty \leq b_{\varphi_2}$. The continuity of $\varphi_2$ on all of $[0, b_{\varphi_2}]$ then ensures that $\varphi_2$ is both bounded and continuous on $[0, \|a\|_\infty]$. Hence in this case also $\varphi_2(|a|) \in \M \subset \widetilde{\M}$.

Next note that both $\psi(\varphi_2(\alpha|a|))$ and $\psi(\alpha\varphi_2(|a|))$ are affiliated to the commutative von Neumann algebra generated by the spectral projections of $|a|$. Moreover the convexity of $\varphi_2$ combined with the fact that $\varphi_2(0) = 0$, reveals that $\varphi_2(\alpha t) \geq \alpha \varphi_2(t)$ for any $t \geq 0$ (since $\alpha \geq 1$). Thus by the Borel functional calculus for affiliated operators, it must follow that  $\psi(\varphi_2(\alpha |a|)) \geq \psi(\alpha \varphi_2(|a|)) \geq 0$. Using this inequality, we may now modify the argument in Lemma \ref{lemma5} to show that since $\psi(\varphi_2(\alpha |a|)) = \varphi_1(\alpha |a|)$ is $\tau$-measurable, $\psi(\alpha \varphi_2(|a|))$ must also be $\tau$-measurable. This inequality then also ensures that $\tau(\psi(\varphi_2(\alpha |a|))) \geq \tau(\psi(\alpha \varphi_2(|a|)))$.

What we have proved above ensures that $$\{1 > \lambda > 0 : \varphi_1\left(\frac{1}{\lambda}|a|\right) \in \widetilde{\M}, \tau\left(\varphi_1\left(\frac{1}{\lambda}|a|\right)\right) \leq 1\}$$ $$\subset  \{\lambda > 0 : \psi\left(\frac{1}{\lambda}\varphi_2(|f|)\right) \in \widetilde{\M}, \tau\left(\psi\left(\frac{1}{\lambda}\varphi_2(|f|)\right)\right) \leq 1\}.$$If now we are given that $\|a\|_{\varphi_1} < 1$, then by the  formula for the Luxemburg-Nakano norm in Proposition \ref{DPvsK}, we must have that $$\|a\|_{\varphi_1} = \inf \{1 > \lambda > 0 : \varphi_1\left(\frac{1}{\lambda}|a|\right) \in \widetilde{\M}, \tau\left(\varphi_1\left(\frac{1}{\lambda}|a|\right)\right) \leq 1\}.$$ Combining this fact with the above inclusion, ensures that 
\begin{eqnarray*}
\|a\|_{\varphi_1} &=& \inf\{1 > \lambda > 0 : \varphi_1\left(\frac{1}{\lambda}|a|\right) \in \widetilde{\M}, \tau\left(\varphi_1\left(\frac{1}{\lambda}|a|\right)\right) \leq 1\}\\
&\geq& \{\lambda > 0 : \psi\left(\frac{1}{\lambda}\varphi_2(|f|)\right) \in \widetilde{\M}, \tau\left(\psi\left(\frac{1}{\lambda}\varphi_2(|f|)\right)\right) \leq 1\}\\
&=& \|\varphi_2(|a|)\|_\psi.
\end{eqnarray*}
\end{proof}

\begin{lemma}\label{lemma4}
Let $\varphi$ be an Orlicz function with $a_\varphi < b_\varphi$, and let $f: [0, \varphi(b_\varphi)] \to [a_\varphi, b_\varphi]$ be the (concave) inverse function of $\varphi$ restricted to $[a_\varphi, b_\varphi]$. Then $$\varphi^{-1}(t) = \left\{
\begin{array}{lll} 
f(t) & \text{if} & 0 \leq t \leq \varphi(b_\varphi)\\
b_\varphi & \text{if} & t > \varphi(b_\varphi) \end{array}\right. .$$Thus $$\varphi \circ \varphi^{-1}(t) = \left\{\begin{array}{lll} 
t & \text{if} & 0 \leq t \leq \varphi(b_\varphi)\\
\varphi(b_\varphi) & \text{if} & t > \varphi( b_\varphi) \end{array}\right. .$$
\end{lemma}

\begin{proof}
Exercise.
\end{proof}

We are now finally ready to prove our main theorem.

\begin{proof}[Proof of the theorem]
By assumption $\tau_2 \circ J$ is $\epsilon - \delta$ absolutely continuous with respect to $\tau_1$. This ensures that $J$ extends uniquely to a Jordan $*$-morphism from $\widetilde{\M_1}$ to $\widetilde{\M_2}$ which is continuous under the topology of convergence in measure \cite[Proposition 4.7]{L1}. We will consistently write $J$ for this extension.

First suppose that $f_J = \frac{d \tau_2\circ J}{d\tau_1} \in L^{\psi^*}(\widetilde{\M_1})$, and let $a \in L^{\varphi_1}(\widetilde{\M_1})$ be given with $a=a^*$ and $\|a\|_{\varphi_1} < 1$. Our first task is to show that then $J(a) \in L^{\varphi_2}(\widetilde{\M_2})$. Now if $b_{\varphi_2} = \infty$, the function $\varphi_2$ will be continuous on all of 
$[0, \infty]$. By approximating with polynomials, we can show that then $J(\varphi_2(|a|)) = \varphi_2(J(|a|))$. If on the other hand $b_{\varphi_2} < \infty$, then also $b_{\varphi_1} < \infty$ (Lemma \ref{lemma1}). So in this case $a \in \M_1$, with $\|J(|a|)\|_\infty \leq \|a\|_\infty < b_{\varphi_1} \leq b_{\varphi_2}$ (Lemma \ref{lemma2}). Thus $\sigma(|a|), \sigma(J(|a|)) \subset [0, b_{\varphi_2})$. The continuity of $\varphi_2$ on $[0, b_{\varphi_2}]$ therefore ensures that $\varphi_2$ will then be continuous and bounded on both $\sigma(|a|)$ and $\sigma(J(|a|))$. With this knowledge we may once again approximate with polynomials and use the functional calculus to conclude that in this case we also have $J(\varphi_2(|a|)) = \varphi_2(J(|a|))$. Noting that $|J(a)| = J(|a|)$ (since $|J(a)|^2 = J(a^2) = J(|a|)^2$), it therefore follows from Proposition \ref{kothe} and Lemma \ref{lemma3} that 
\begin{eqnarray*}
\tau_2(\varphi_2(|J(a)|)) &=& \tau_2(\varphi_2(J(|a|)))\\
&=& \tau_2(J(\varphi_2(|a|)))\\
&=& \tau_1(f_J^{1/2}\varphi_2(|a|)f_J^{1/2})\\
&=& \tau_1(f_J\varphi_2(|a|))\\
&\leq& \|f_J\|^0_{\psi^*}\|\varphi_2(|a|)\|_\psi\\
&\leq& \|f_J\|^0_{\psi^*}\|a\|_{\varphi_1}\\
&<& \infty.
\end{eqnarray*} 
(The fourth equality in the above computation follows from \cite[Proposition 5.2]{DDdP3}.) Thus $J$ maps the self-adjoint portion of $L^{\varphi_1}(\widetilde{\M_1})$ (and hence all of $L^{\varphi_1}(\widetilde{\M_1})$) into $L^{\varphi_2}(\widetilde{\M_2})$. By the Closed Graph Theorem this is enough to ensure that $J$ restricts to a bounded operator $C_J$ from $L^{\varphi_1}(\widetilde{\M_1})$ to $L^{\varphi_2}(\widetilde{\M_2})$. (If for some sequence $\{a_n\} \subset L^{\varphi_1}(\widetilde{\M_1})$ we have that $a_n \to a$ and $C_J(a_n) = J(a_n) \to b$ with respect to the ambient Orlicz topologies, then $a_n \to a$ and $J(a_n) \to b$ with respect to the measure topologies as well \cite{DDdP}. The fact that $J$ acts continuously from $\widetilde{\M_1}$ to $\widetilde{\M_2}$ then ensures that $J(a) = b$.)

We proceed to compute a more exact estimate of the norm of $C_J$ restricted to the self-adjoint portion of $L^{\varphi_1}(\widetilde{\M_1})$. As before let $a \in L^{\varphi_1}(\widetilde{\M_1})$ be given with $a=a^*$ and $\|a\|_{\varphi_1} < 1$. For any $\mu \geq \mathrm{max}\{1, \|f_J\|^0_{\psi^*}\}$, Lemma \ref{lemma2} ensures that $\varphi_2(\frac{1}{\mu}|J(a)|) \in \widetilde{\M_2}$ with $\tau_2(\varphi_2(\frac{1}{\mu}|J(a)|)) \leq \tau_2(\frac{1}{\mu}\varphi_2(|J(a)|))$. An application of Proposition \ref{kothe} and Lemma \ref{lemma3}, then shows that
\begin{eqnarray*}
\tau_2(\varphi_2(\frac{1}{\mu}|J(a)|)) &\leq& \tau_2(\frac{1}{\mu}\varphi_2(|J(a)|))\\
&=& \frac{1}{\mu}\tau_2(\varphi_2(J(|a|)))\\
&=& \frac{1}{\mu}\tau_2(J(\varphi_2(|a|)))\\
&=& \frac{1}{\mu}\tau_1(f_J\varphi_2(|a|))\\
&\leq& \frac{\|f_J\|^0_{\psi^*}}{\mu}.\|\varphi_2(|a|)\|_\psi\\
&\leq& 1.\|a\|_{\varphi_1}\\
&<& 1.
\end{eqnarray*}
Thus $\mu \in \{\lambda > 0 : \varphi_2(\frac{1}{\lambda}|J(a)|) \in \widetilde{\M_2},  \tau_2(\varphi_2(\frac{1}{\lambda}|J(a)|)) \leq 1\}$ whenever $\mu \geq \mathrm{max}\{1, \|f_J\|^0_{\psi^*}$. From the formula for the norm in Proposition \ref{DPvsK}, this clearly forces $\|C_J(a)\|_{\varphi_2} = \|J(a)\|_{\varphi_2} \leq \mathrm{max}\{1, \|f_J\|^0_{\psi^*}\}$. The claim follows. 

Conversely assume that the canonical extension of $J$ to a Jordan $*$-morphism from $\widetilde{\M_1}$ to $\widetilde{\M_2}$, restricts to a bounded map $C_J$ from $L^{\varphi_1}(\widetilde{\M_1})$ to $L^{\varphi_2}(\widetilde{\M_2})$, and that $\varphi_2$ satisfies $\Delta_2$ for all $t$. Then of course $a_{\varphi_2} = 0$ and $b_{\varphi_2} = \infty$. By Lemma \ref{lemma4} this fact means in particular that $\varphi_2$ and $\varphi^{-1}_2$ are proper inverses of each other which are continuous on all of $[0, \infty]$. Our task is to show that the above conditions force $f_J \in L^{\psi^*}(\widetilde{\M_1})$. By Proposition \ref{kothe} this will follow if we can show that $f_J \in \widetilde{\M_1}$, and that $f_Ja \in L^1(\M_1, \tau_1)$ for each $a \in L^{\psi}(\widetilde{\M_1})$. 

The first step in verifying these facts, is to show that the canonical extension of $J$ maps 
$L^\psi(\widetilde{\M_1})$ into $L^1(\widetilde{\M_2})$. To this end let $a \in L^\psi(\widetilde{\M_1})$ be given with $a \geq 0$. By scaling $a$ if necessary, we may assume without loss of generality that $\|a\|_\psi < 1$. Since $\varphi_2$ and $\varphi^{-1}_2$ are continuous on all of $[0, \infty]$, it is clear that $\varphi_2^{-1}(a) \in \widetilde{\M_1}$ and $\varphi_2^{-1}(J(a)) \in \widetilde{\M_2}$, with $$\varphi_2^{-1}(J(a)) = J(\varphi_2^{-1}(a)), \qquad \varphi_2(\varphi_2^{-1}(J(a))) = J(a),$$ $$\varphi_1(\varphi_2^{-1}(a)) = \psi(\varphi_2\circ\varphi_2^{-1}(a)) = \psi(a)$$by the Borel functional calculus. By Lemma 
\ref{lemma2}, the assumption $\|a\|_\psi < 1$ ensures that $\tau_1(\psi(a)) < 1$. In other words $\tau_1(\varphi_1(\varphi_2^{-1}(a))) < 1$. Therefore $\varphi_2^{-1}(a) \in L^{\varphi_1}(\widetilde{\M_1})$. But then $\varphi_2^{-1}(J(a)) = J(\varphi_2^{-1}(a)) = C_J(\varphi_2^{-1}(a)) \in L^{\varphi_2}(\widetilde{\M_2}).$ By Lemma \ref{lemma5}, this ensures that 
$$\tau_2(J(a)) = \tau_2(\varphi_2(\varphi_2^{-1}(J(a)))) < \infty$$ and hence that $J(a) \in L^1(\M_2, \tau_2)$. Thus the canonical extension of $J$ maps the positive part of $L^\psi(\widetilde{\M_1})$ into $L^1(\widetilde{\M_2})$. But since $L^\psi(\widetilde{\M_1})_+$ spans all of $L^\psi(\widetilde{\M_1})$, it is trivial to conclude that $J$ maps all of $L^\psi(\widetilde{\M_1})$ into $L^1(\widetilde{\M_2})$.
As before the fact that $L^\psi(\widetilde{\M_1})$ and $L^1(\widetilde{\M_2})$ respectively embed continuously into 
$\widetilde{\M_1}$ and $\widetilde{\M_2}$, coupled with the fact $J$ acts continuously from $\widetilde{\M_1}$ to $\widetilde{\M_2}$, is enough to ensure that in its action from $L^\psi(\widetilde{\M_1})$ to $L^1(\widetilde{\M_2})$, the restriction of $J$ has a closed graph. Thus by the Closed Graph Theorem there must exist a constant $K > 0$ so that $$\tau_2(|J(a)|) \leq K\|a\|_\psi \quad\mbox{for all}\quad a \in L^\psi(\widetilde{\M_1}).$$

Let $b \in \M_1^+ $ be given with $\tau_1(b) < \infty$. We want to show that then $\tau_2(J(b)) < \infty$. By suitably scaling $b$ if necessary, we may assume that $\|b\|_\infty < b_\psi$. Since $\psi$ is continuous on $[0, b_\psi]$, it is then both convex and bounded on $[0, \|b\|_\infty]$. Thus we may select $k > 0$ so that $$\psi(t) \leq kt \quad\mbox{for all}\quad 0 \leq  t \leq \|b\|_\infty.$$(Any line-segment from the origin to a point 
$(\|b\|_\infty, q)$ with $\psi(\|b\|_\infty) < q$ will do.) By the Borel functional calculus we will then have that $$0 \leq \psi(b) \leq kb,$$which in turn ensures that 
$$\tau_1(\psi(b)) \leq k\tau_1(b) < \infty.$$Thus $b \in L^\psi(\widetilde{\M_1})$. But since the canonical extension of $J$ maps $L^\psi(\widetilde{\M_1})$ into $L^1(\widetilde{\M_2})$, we must have $$\tau_1(f_J^{1/2}bf_J^{1/2}) = \tau_2(J(b)) < \infty.$$By Proposition 6.5 of 
\cite{PT}, this fact is sufficient to ensure that for any spectral projection of $f_J$ of the form $e_{[\lambda, \infty)}$, we will have $\tau_1(f_Je_{[\lambda, \infty)}) < \infty$ for $\lambda > 0$ large enough. But since $\lambda e_{[\lambda, \infty)} \leq f_Je_{[\lambda, \infty)}$, this in turn ensures that $\tau_1(e_{[\lambda, \infty)}) < \infty$ for $\lambda > 0$ large enough. In other words $f_J$ is $\tau_1$-measurable.

Finally let $b \in \M_1 \cap L^1(\M_1, \tau_1)$ be given with $\|b\|_\psi \leq 1$, and let $e_n = e_{[o,n]}$ be the spectral projection from the spectral resolution of $f_J$ corresponding to the interval $[0, n]$. Let $v$ be the partial isometry in the polar decomposition $f_J^{1/2}e_nb = v|f_J^{1/2}e_nb|$. Then $e_nbv^*e_n \in L^\psi(\widetilde{\M_1})$ with $\|e_nbv^*e_n\|_\psi \leq \|b\|_\psi \leq 1$ (see the discussion following Definition 2.1 of \cite{DDdP3}). Using the fact that $\tau_1(xy) = \tau_1(yx)$ for $x \in \M_1, y \in \M_1 \cap L^1(\M_1, \tau_1)$, we may conclude that 
\begin{eqnarray*}
\tau_1(|f_Je_nb|) &=& \tau_1((v^*f_J^{1/2}e_n)(f_J^{1/2}e_nb))\\
&=& \tau_1((f_J^{1/2}e_nb)(v^*e_nf_J^{1/2}))\\
&=& \tau_2(J(e_nbv^*e_n))\\
&\leq& \tau_2(|J(e_nbv^*e_n)|)\\
&\leq& K\|e_nbv^*e_n\|_\psi\\
&\leq& K 
\end{eqnarray*}
So by Proposition 5.3(ii) of \cite{DDdP3}, we have that $f_Je_n \in L^{\psi^*}(\widetilde{\M_1})$, with $\|f_Je_n\|^0_{\psi^*} 
\leq K$. But then $\|f_J\|^0_{\psi^*} \leq K$ by \cite[Proposition 5.4(ii)]{DDdP3}. Thus as required, $f_J \in L^{\psi^*}(\widetilde{\M_1})$ by \cite[Proposition 5.3(ii)]{DDdP3}.
\end{proof}

In closing we make a final comment regarding the significance of the $\Delta_2$ condition in this context.

\begin{remark}
Let $\M$ be semifinite von Neumann algebra with fns trace $\tau$, and let $L^\rho(0, \infty)$ be a classical Banach Function Space on $[0, \infty)$. It is clear from  \cite[1.3.8]{BS} that this space will have an absolutely continuous norm in the sense of \cite{BS} if and only if it has an order continuous norm in the sense of \cite{DDdP3}. Now consider the specific case where for some Orlicz function $\varphi$, $L^\rho(0, \infty) = L^\varphi(0, \infty)$ is the associated Orlicz space. If we combine the above observation with the discussion on p 96 of \cite{KR}, we see that (at least for the case of Young's functions) $\varphi$ satisfies $\Delta_2$ for all $t$ if and only if $L^\varphi(0, \infty)$ has order continuous norm. But by \cite[Proposition 3.6]{DDdP3}, $L^\varphi(\widetilde{\M})$ will have order continuous norm whenever $L^\varphi(0, \infty)$ has order continuous norm. Thus the $\Delta_2$ condition is intimately related to the question of whether $L^\varphi(\widetilde{\M})$ has order continuous norm.

The presence of an order continuous norm in turn puts us in a position where we can try approximate each element of $L^\varphi(\widetilde{\M})$ by ``simple functions''. Given some positive element $a$ of $L^\varphi(\widetilde{\M})$, the idea is to try and find a sequence $\{a_n\}$ of Riemann sums of the form $a_n = \sum_{k=1}^m \lambda_ke_k$ for which $a - a_n$ decreases to 0 in $\widetilde{\M}$. (For each  $a_n = \sum_{k=1}^m \lambda_ke_k$, the $e_k$'s are mutually orthogonal projections from the spectral resolution of $a$ for which $\tau(e_k) < \infty$.)  The presence of an order continuous norm on $L^\varphi(\widetilde{\M})$, then ensures that $a_n \to a$ in the $\|\cdot\|_\varphi$ norm.
\end{remark}

\section{Acknowledgments}

The support of a grant under the Poland-South Africa Cooperation Agreement is gratefully acknowledged.

\end{document}